\documentclass[11pt,reqno]{amsart}

\usepackage{amssymb, amsmath, amsthm, mathrsfs, MnSymbol}
\usepackage{hyperref}
\usepackage{enumitem}
\usepackage{float}

\newtheorem{theo}{Theorem}[section]
\newtheorem{lem}{Lemma}[section]
\newtheorem{cor}{Corollary}[section]

\newtheorem{exm}{Example}[section]

\newtheorem{rem}{Remark}[section]

\newtheorem*{theoA}{Theorem A}
\newtheorem*{theoB}{Theorem B}
\newtheorem*{theoC}{Theorem C}

\newtheorem*{remA}{Remark A}

\newcommand{\ol}{\overline}
\newcommand{\be}{\begin{equation}}
	\newcommand{\ee}{\end{equation}}
\newcommand{\beas}{\begin{eqnarray*}}
	\newcommand{\eeas}{\end{eqnarray*}}
\newcommand{\bea}{\begin{eqnarray}}
	\newcommand{\eea}{\end{eqnarray}}

\numberwithin{equation}{section}

\begin{document}

\title[T\MakeLowercase{he} Y\MakeLowercase{ang}-H\MakeLowercase{ua theorems in Several}....]{\LARGE T\Large\MakeLowercase{he} Y\MakeLowercase{ang}-H\MakeLowercase{ua} \MakeLowercase{theorems in Several Complex Variables}}

\date{}
\author[A. B\MakeLowercase {anerjee}, S. M\MakeLowercase{ajumder}, D. P\MakeLowercase{ramanik} \MakeLowercase{and} N. S\MakeLowercase{arkar}]{A\MakeLowercase {bhijit} B\MakeLowercase {anerjee}, S\MakeLowercase{ujoy} M\MakeLowercase{ajumder}, D\MakeLowercase{ebabrata} P\MakeLowercase{ramanik} \MakeLowercase{and} N\MakeLowercase{abadwip} S\MakeLowercase{arkar}$^{*}$}
\address{ Department of Mathematics, University of Kalyani, West Bengal 741235, India.}
\email{abanerjee\_kal@yahoo.co.in}
\address{Department of Mathematics, Raiganj University, Raiganj, West Bengal-733134, India.}
\email{sm05math@gmail.com, sjm@raiganjuniversity.ac.in}
\address{Department of Mathematics, Raiganj University, Raiganj, West Bengal-733134, India.}
\email{debumath07@gmail.com}
\address{Department of Mathematics, Raiganj University, Raiganj, West Bengal-733134, India.}
\email{naba.iitbmath@gmail.com}

\renewcommand{\thefootnote}{}
\footnote{*\emph{Corresponding Author}: Nabadwip Sarkar.}

\renewcommand{\thefootnote}{\arabic{footnote}}
\setcounter{footnote}{0}

\begin{abstract} In this paper, we investigate meromorphic solutions in $\mathbb{C}^m$ of the nonlinear differential equation \[\displaystyle f^n\partial_u(f)g^n\partial_u(g)=1,\]
	where $\partial_u(f)=\sum_{j=1}^mu_j\partial_j(f)$ and $\sum_{j=1}^m u_j\neq 0$. Our results extend those of Yang and Hua \cite{YH1} to the framework of several complex variables. Moreover, we establish new uniqueness theorems that further generalize their conclusions to higher dimensions. As an application, explicit solutions of certain nonlinear partial differential equations in several variables are derived, and their physical interpretations are summarized in tabular form.  
\end{abstract}

\thanks{Typeset by \AmS -\LaTeX}
\maketitle

\section{\textbf{Introduction and Historical Background}}

	Solving nonlinear differential equations remains one of the most challenging and stimulating problems in mathematical analysis. Among the earliest contributions to this field, Yang and Hua \cite{YH1} were the first to investigate the nonlinear differential equation
	\begin{equation}\label{1.1}
		f^n f^{(1)} g^n g^{(1)} = 1,
	\end{equation}
	and to determine possible functional forms for the solutions $f$ and $g$. Their pioneering work laid a foundation that has inspired extensive developments in the study of meromorphic functions and nonlinear differential equations.

\subsection*{\textbf{Historical Results}}

	In 1997, Yang and Hua~\cite{YH1} presented a complete characterization of the non-constant entire and meromorphic solutions of equation~\eqref{1.1}. Their main results can be summarized as follows:

\begin{theoA}[\cite{YH1}, Theorem 3]
	Let $f$ and $g$ be two non-constant entire functions in $\mathbb{C}$, and let $n$ be a positive integer. 
	If $f^n f^{(1)} g^n g^{(1)} = 1$, then
	\[
	f(z) = c_1 \exp(-cz) \quad \text{and} \quad g(z) = c_2 \exp(cz),
	\]
	where $c$, $c_1$, and $c_2$ are non-zero constants satisfying 
	$c^2 (c_1 c_2)^{\,n+1} = -1$.
\end{theoA}

\begin{theoB}[\cite{YH1}, Theorem 2]
	Let $f$ and $g$ be two non-constant meromorphic functions in $\mathbb{C}$, and let $n \ge 6$ be an integer. 
	If $f^n f^{(1)} g^n g^{(1)} = 1$, then
	\[
	f(z) = c_1 \exp(-cz) \quad \text{and} \quad g(z) = c_2 \exp(cz),
	\]
	where $c$, $c_1$, and $c_2$ are non-zero constants satisfying 
	$c^2 (c_1 c_2)^{\,n+1} = -1$.
\end{theoB}

\subsection*{\textbf{Uniqueness of Meromorphic Functions}}

	These results have significantly influenced the study of uniqueness problems for meromorphic functions, particularly in the context of nonlinear differential polynomials sharing a common value. In the same work, Yang and Hua also explored the uniqueness of meromorphic functions when two linear differential polynomials share a common $a$-point, for $a \in \mathbb{C}\backslash \{0\}$. The following theorem captures their key result:

\begin{theoC}[\cite{YH1}, Theorem 1]
	Let $f$ and $g$ be two non-constant meromorphic functions in $\mathbb{C}$, $n \ge 11$ be an integer, and $a \in \mathbb{C}\backslash  \{0\}$. 
	If $f^n f^{(1)}$ and $g^n g^{(1)}$ share the value $a$ counting multiplicities (CM), then either
	\[
	f = d g \quad \text{for some $(n+1)$-th root of unity } d,
	\]
	or
	\[
	g(z) = c_1 \exp(cz), \quad f(z) = c_2 \exp(-cz),
	\]
	where $c$, $c_1$, and $c_2$ are constants satisfying 
	\[
	(c_1 c_2)^{\,n+1} c^2 = -a^2.
	\]
\end{theoC}

\begin{remA}
	If $f$ and $g$ are entire, it suffices to assume $n \ge 7$ in Theorem C.
\end{remA}

\subsection*{\textbf{Modern Context and Multi-variable Generalizations}}

	\begin{itemize}
		\item Nevanlinna theory in several complex variables provides a framework for studying value distribution and uniqueness phenomena in higher dimensions.
		\item Applications include: complex geometry, normal families, PDEs, partial difference equations, and Fermat-type functional equations (see \cite{TBC1}--\cite{HZ1}, \cite{ps1}--\cite{LZ}, \cite{MD1}, \cite{MDP}, \cite{MSP}, \cite{GS2}, \cite{XC1}--\cite{XW1}).
	\end{itemize}

\medskip

Let 
\[
S_{m} = \{ (u_1, u_2, \ldots, u_m) \in \mathbb{C}^{m} : u_1 + u_2 + \cdots + u_m \neq 0 \}.
\]
For $u \in S_m$, we define
\[
\partial_u(f) = \sum_{j=1}^m u_j \frac{\partial f}{\partial z_j} \not\equiv 0.
\]
The $k$-th order derivative $\partial_u^k(f)$ is defined inductively by 
\[
\partial_u^k(f) = \partial_u(\partial_u^{\,k-1}(f)).
\]

\subsection*{\textbf{Problem Statement}}

	In this paper, we focus on the following nonlinear partial differential equation:
	\begin{equation}\label{2.1}
		f^n \partial_u(f) \, g^n \partial_u(g) = 1,
	\end{equation}
	where $f$ and $g$ are non-constant meromorphic functions in $\mathbb{C}^m$, and $n$ is a positive integer.
	
	Furthermore, we investigate the uniqueness problem for meromorphic functions in $\mathbb{C}^m$ with respect to nonlinear partial differential polynomials generated by them sharing a common value.

\section*{\textbf{Brief Discussion of the Theory of Several Complex Variables}}

	We begin by recalling some basic notions and definitions from the theory of several complex variables (see, for example, \cite{HLY1, WS2}). The purpose of this section is to establish the foundational notations and concepts that will be used throughout the paper, particularly those relevant to Nevanlinna theory in higher dimensions. For clarity and completeness, we also include a concise conceptual breakdown to help illuminate the central ideas and their interrelations within this framework.

	\section*{Basic Sets and Notations}
	
	\begin{itemize}[leftmargin=1.3cm]
		\item $\displaystyle \mathbb{Z}_+ = \{n \in \mathbb{Z} : n \ge 0\}$ --- the set of nonnegative integers.
		\item $\displaystyle \mathbb{Z}^+ = \{n \in \mathbb{Z} : n > 0\}$ --- the set of positive integers.
		\item For any set $S$, $\#(S)$ denotes the \emph{cardinality} (number of elements) of $S$.
	\end{itemize}
	
	\section*{Differential Forms on $\mathbb{C}^m$}
	
	\begin{itemize}[leftmargin=1.3cm]
		\item The exterior derivative decomposes as
		$d = \partial + \bar{\partial}$.
		\item The twisted differential is defined by
		\[
		d^c = \frac{i}{4\pi}(\bar{\partial} - \partial),
		\]
		and hence
		\[
		dd^c = \frac{i}{2\pi} \partial \bar{\partial}.
		\]
	\end{itemize}
	
	\section*{Exhaustion and K\"ahler Structure}
	
	\begin{itemize}[leftmargin=1.3cm]
		\item A smooth, nonnegative function $\tau : \mathbb{C}^m \to [0, b)$ ($0 < b \le \infty$) is an \emph{exhaustion function} if $\tau^{-1}(K)$ is compact for every compact $K$.
Standard exhaustion:
		$\tau_m(z) = \|z\|^2$.
		\item The standard K\"ahler metric on $\mathbb{C}^m$ is given by
		$\upsilon_m = dd^c \tau_m > 0$.
		
		\item On $\mathbb{C}^m\backslash \{0\}$, define
		$\omega_m = dd^c \log \tau_m$ and
		$\sigma_m = d^c \log \tau_m \wedge \omega_m^{m-1}$.
		\item For $S \subseteq \mathbb{C}^m$ and $r>0$, denote:
		$S[r]$, $S(r)$ and $S\langle r\rangle$
		as the intersections of $S$ with the closed ball, open ball, and sphere of radius $r$, respectively, centered at $0 \in \mathbb{C}^m$.
	\end{itemize}
	
	\section*{Analytic Sets}
	
	\begin{itemize}[leftmargin=1.3cm]
		\item A closed subset $A \subset G$ is called \emph{analytic} if for every $a \in A$ there exist holomorphic functions $f_1, \dots, f_l$ defined in a neighbourhood $N(a)$ such that
		\[
		A \cap N(a) = \{z \in N(a) : f_1(z) = \cdots = f_l(z) = 0\}.
		\]
		\item A point $a \in A$ is \emph{regular} if $df_1(a), \dots, df_l(a)$ are linearly independent.
		Otherwise, $a$ is a \emph{singular point}.
		Denote:
		$S(A) = \text{set of singular points}$ and $R(A) = A\backslash  S(A)$.
	\end{itemize}
	
	\section*{Multiplicities of Zeros}
	
	\begin{itemize}[leftmargin=1.3cm]
		\item For a holomorphic function $f$ on an open set $G \subset \mathbb{C}^m$ and multi-index $I = (i_1, \ldots, i_m)$, define
		\[
		\partial^I f(z) = 
		\frac{\partial^{|I|} f(z)}{\partial z_1^{i_1} \cdots \partial z_m^{i_m}},
		\qquad |I| = \sum_{j=1}^m i_j.
		\]
		\item The \emph{zero multiplicity} $\mu_f^0(a)$ at $a \in \mathbb{C}^m$ is the smallest integer $i$ such that the homogeneous term $P_i(z-a)$ in the Taylor expansion
		\[
		f(z) = \sum_{i=0}^\infty P_i(z-a)
		\]
		is nonzero:
		\[
		\mu_f^0(a) = \min\{i : P_i(z-a) \not\equiv 0\}.
		\]
	\end{itemize}
	
	\section*{Meromorphic Functions and Divisors}
	
	\begin{itemize}[leftmargin=1.3cm]
		\item If $f$ is meromorphic on $G$, there exist holomorphic $g, h$ such that $hf = g$ on $G$, and $g$, $h$ are coprime at each point.
		The \emph{$c$-multiplicity} of $f$ is defined by
		\[
		\mu_f^c =
		\begin{cases}
			\mu_{g - ch}^0, & c \in \mathbb{C},\\[6pt]
			\mu_h^0, & c = \infty.
		\end{cases}
		\]
		\item The \emph{divisor} of $f$ is
		$\nu = \mu_f = \mu_f^0 - \mu_f^\infty$,
		and its support is
		\[
		\operatorname{supp}(\nu) = \overline{\{z \in G : \nu(z) \ne 0\}}.
		\]
	\end{itemize}
	
	\section*{Counting and Valence Functions}
	
	\begin{itemize}[leftmargin=1.3cm]
		\item For $A = \operatorname{supp}(\nu)$ and $t>0$, the \emph{counting function} is
		\[
		n_\nu(t) = t^{-2(m-1)} \int_{A[t]} \nu \, \upsilon_m^{m-1}.
		\]
		\item The \emph{valence function} is
		\[
		N_\nu(r) = N_\nu(r, r_0) = \int_{r_0}^r n_\nu(t) \frac{dt}{t}, \qquad (r \ge r_0).
		\]
		\item For a meromorphic function $f$:
		\[
		n_{\mu_f^a}(t) =
		\begin{cases}
			n(t, a; f), & a \in \mathbb{C},\\
			n(t, f), & a = \infty,
		\end{cases}
		\quad
		N_{\mu_f^a}(r) =
		\begin{cases}
			N(r, a; f), & a \in \mathbb{C},\\
			N(r, f), & a = \infty.
		\end{cases}
		\]
		\item For $k \in \mathbb{N}$, define the \emph{truncated multiplicity functions}:
		\[
		\mu_{f,k}^a(z) = \min\{\mu_f^a(z), k\}.
		\]
		Then
		\[
		n_k(t, a; f), \quad N_k(r, a; f)
		\]
		denote the corresponding truncated counting functions, and for $k=1$:
		\[
		\overline{n}(t, a; f) = n_1(t, a; f), \qquad
		\overline{N}(r, a; f) = N_1(r, a; f).
		\]
	\end{itemize}
	\section*{Sharing of Values and Proximity Functions}
	
	\begin{itemize}[leftmargin=1.3cm]
		
		\item Let $f$, $g$, and $a$ be meromorphic functions on $\mathbb{C}^m$.  
		Then one can find three pairs of entire functions
		$(f_1, f_2)$, $(g_1, g_2)$ and $(a_1, a_2)$,
		where each pair is coprime at every point in $\mathbb{C}^m$, such that
		\[
		f = \frac{f_2}{f_1}, \qquad g = \frac{g_2}{g_1}, \qquad a = \frac{a_2}{a_1}.
		\]
		\item We say that $f$ and $g$ \emph{share the value $a$ by counting multiplicities (CM)} if
		\[
		\mu_{a_1 f_2 - a_2 f_1}^0 = \mu_{a_1 g_2 - a_2 g_1}^0 \quad (a \not\equiv \infty),
		\]
		and
		\[
		\mu_{f_1}^0 = \mu_{g_1}^0 \quad (a = \infty).
		\]
		
		\item A property is said to hold \emph{generically} if it holds except on an analytic subset of strictly smaller dimension.  
		For instance, if $f$ is a meromorphic function on $\mathbb{C}^m$, then there exist holomorphic functions $g$ and $h$ such that
		\[
		\dim\big(g^{-1}(\{0\}) \cap h^{-1}(\{0\})\big) \le m - 2,
		\qquad f = \frac{g}{h}.
		\]
		Consequently, the common zero factor of $g$ and $h$ can be cancelled \emph{generically} over its support sets (see \cite[pp.~103]{HLY1}).
		
		\item With the help of the positive logarithm function, we define the \emph{proximity function} of $f$ by
		\[
		m(r, f) = \int_{\mathbb{C}^m\langle r\rangle} \log^+ |f| \, \sigma_m \ge 0.
		\]
		
		\item The \emph{characteristic function} of $f$ is then defined as
		$T(r, f) = m(r, f) + N(r, f)$.
		
		\item We further define
		\[
		m(r, a; f) =
		\begin{cases}
			m(r, f), & a = \infty, \\[6pt]
			m\!\left(r, \dfrac{1}{f - a}\right), & a \in \mathbb{C}.
		\end{cases}
		\]
		
		\item The \emph{First Main Theorem of Nevanlinna Theory} in this setting states that, for $a \in \mathbb{C}$,
		\[
		m(r, a; f) + N(r, a; f) = T(r, f) + O(1),
		\]
		where $O(1)$ denotes a bounded function for sufficiently large $r$.
	\end{itemize}
\medskip
The remainder of the paper is organized as follows:\par  
In Section~2, we find possible solutions of the nonlinear partial differential equation~(\ref{2.1}) 
for the functions $f$ and $g$. In Section~3, we consider the uniqueness problem for meromorphic functions in $\mathbb{C}^m$ 
concerning nonlinear partial differential polynomials generated by them sharing one value.

\section{\bf{Solutions of $f^n\partial_u(f) g^n\partial_u(g)=1$}}
The differential equation  
\[
f^n\,\partial_u(f)\, g^n\,\partial_u(g) = 1
\]
can be regarded as a natural extension, to several complex variables, of classical nonlinear differential equations involving a function and its derivative in one variable, such as  
\[
f^n f' = 1, \qquad f^n f' g^n g' = 1, \qquad \text{or more generally } f^p (f')^q = c.
\]
Equations of this type have been extensively studied in connection with the growth, uniqueness, and invariance properties of meromorphic functions. For instance, the equation \(f^n f' = 1\) in one variable implies that \(f\) is of exponential type, namely
\[
f(z) = c_1 e^{cz},
\]
where \(c_1, c \neq 0\).  

In several variables, the directional derivative \(\partial_u(f)\) serves as the analogue of the ordinary derivative, leading to more intricate interactions among variables. The above equation links \(f\) and \(g\) through their powers and directional derivatives, imposing strong structural constraints on possible solutions.  

Our objective is to determine all non-constant entire or meromorphic pairs \((f, g)\) satisfying this identity. The following theorems describe the general form of these solutions.

\begin{theo}\label{t2.1} Let $f:\mathbb{C}^m\to \mathbb{C}$ and $g:\mathbb{C}^m\to \mathbb{C}$ be two non-constant entire functions and let $n\geq 1$ be an integer. If 
\[(f^n\partial_uf)(g^n\partial_ug)\equiv 1,\]
then $f=\exp(\alpha)$ and $g=\exp(\beta)$, where $\alpha$ and $\beta$ are non-constant entire functions in $\mathbb{C}^m$ such that $\partial_u(\alpha)=c$, $\partial_u(\beta)=-c$ and $c^2e^{(n+1)(\alpha+\beta)}=-1$. 
\end{theo}
\begin{theo}\label{t2.2} Let $f:\mathbb{C}^m\to \mathbb{P}^1$ and $g:\mathbb{C}^m\to \mathbb{P}^1$ be two non-constant meromorphic functions and let $n\geq 6$ be an integer. 
If 
\[f^n\partial_u(f)g^n\partial_u(g)\equiv 1,\]
then the conclusions of Theorem \ref{t2.1} hold.
\end{theo}

Let $u=(u_1,\ldots,u_m)$ such that $u_j=1$ and $u_i=0$ for $i\neq j$. Clearly $\partial_u(f)=\partial_j(f)$ for $m>1$ and $\partial_u(f)=f^{(1)}$ for $m=1$ and $u_1=1$. Then from Theorem \ref{t2.2}, we get the following.

\begin{cor}\label{c2.1} Let $f:\mathbb{C}^m\to \mathbb{P}^1$ and $g:\mathbb{C}^m\to \mathbb{P}^1$ be two non-constant meromorphic functions and let $n\geq 6$ be an integer. If 
\[f^n\partial_j(f)g^n\partial_j(g)\equiv 1,\]
where $j\in\mathbb{Z}[1,m]$, then $f=\exp(\alpha)$ and $g=\exp(\beta)$, where $\alpha$ and $\beta$ are non-constant entire functions in $\mathbb{C}^m$ such that $\partial_j(\alpha)=c$, $\partial_j(\beta)=-c$ and $c^2e^{(n+1)(\alpha+\beta)}=-1$. 
\end{cor}

\begin{cor}\label{c2.2} Let $f:\mathbb{C}^m\to \mathbb{P}^1$ and $g:\mathbb{C}^m\to \mathbb{P}^1$ be two non-constant meromorphic functions and let $n\geq 6$ be an integer. If 
\[f^n\partial_j(f)g^n\partial_j(g)\equiv 1,\]
for all $j\in\mathbb{Z}[1,m]$, then $f(z)=c_1\exp(c(z_1+z_2+\ldots+z_m))$ and $g(z)=c_2\exp(-c(z_1+z_2+\ldots+z_m))$, where $c$, $c_1$ and $c_2$ are non-zero constants such that $c^2(c_1c_2)^{n+1}=-1$.
\end{cor}
\begin{table}[H]
	\centering
	\renewcommand{\arraystretch}{1.5}
	\begin{tabular}{|p{2.5cm}|p{2.5cm}|p{2cm}|p{3.5cm}|p{3cm}|}
		\hline
		\textbf{Result} & \textbf{Domain / Type of Functions} & \textbf{Condition on $n$} & \textbf{Equation Assumed} & \textbf{Conclusions} \\
		\hline
		\textbf{Theorem \ref{t2.1}} &
		$f,g:\mathbb{C}^m \to \mathbb{C}$, non-constant entire functions &
		$n \ge 1$ &
		$(f^n\partial_u f)(g^n\partial_u g) \equiv 1$ &
		$f = e^{\alpha},\ g = e^{\beta}$, where $\alpha,\beta$ are non-constant entire functions such that $\partial_u(\alpha) = c$, $\partial_u(\beta) = -c$, and $c^2 e^{(n+1)(\alpha+\beta)} = -1$. \\
		\hline
		\textbf{Theorem \ref{t2.2}} &
		$f,g:\mathbb{C}^m \to \mathbb{P}^1$, non-constant meromorphic functions &
		$n \ge 6$ &
		$f^n\partial_u(f)g^n\partial_u(g)\equiv 1$ &
		Same conclusions as Theorem \ref{t2.1}. \\
		\hline
		\textbf{Corollary \ref{c2.1}} &
		$f,g:\mathbb{C}^m \to \mathbb{P}^1$, non-constant meromorphic functions &
		$n \ge 6$ &
		$f^n\partial_j(f)g^n\partial_j(g)\equiv 1$ (for a fixed $j \in \mathbb{Z}[1,m]$) &
		$f = e^{\alpha},\ g = e^{\beta}$, where $\alpha,\beta$ are non-constant entire functions such that $\partial_j(\alpha)=c$, $\partial_j(\beta)=-c$, and $c^2 e^{(n+1)(\alpha+\beta)} = -1$. \\
		\hline
		\textbf{Corollary \ref{c2.2}} &
		$f,g:\mathbb{C}^m \to \mathbb{P}^1$, non-constant meromorphic functions &
		$n \ge 6$ &
		$f^n\partial_j(f)g^n\partial_j(g)\equiv 1$ for all $j \in \mathbb{Z}[1,m]$ &
		$f(z)=c_1 e^{c(z_1+z_2+\cdots+z_m)}$, $g(z)=c_2 e^{-c(z_1+z_2+\cdots+z_m)}$, where $c,c_1,c_2 \neq 0$ and $c^2(c_1c_2)^{n+1} = -1$. \\
		\hline
	\end{tabular}\vspace{1cc}
	\caption{Comparison of Theorems \ref{t2.1}, \ref{t2.2} and Corollaries \ref{c2.1}, \ref{c2.2}.}
\end{table}


\subsection {{\bf Auxiliary lemmas}}
In the proof of Theorems \ref{t2.1}-\ref{t2.2}, we use of the following key lemmas.
\begin{lem}\label{l2}\cite[Lemma 1.37]{HLY1} Let $f:\mathbb{C}^m\to\mathbb{P}^1$ be a non-constant meromorphic function and let  $I=(\alpha_1,\alpha_2,\ldots,\alpha_m)\in \mathbb{Z}^m_+$ be a multi-index. Then for any $\varepsilon>0$, we have
\[\parallel\;m\left(r,\partial^I(f)/f\right)\leq |I|\log^+T(r,f)+|I|(1+\varepsilon)\log^+\log T(r,f)+O(1)\]
where $\parallel$ indicates that the inequality holds only outside a set of finite measure on $\mathbb{R}^+$
\end{lem}

\begin{lem}\label{l7a}\cite[Lemma 1.68]{HLY1} Let $f_1:\mathbb{C}^m\to\mathbb{P}^1$ and $f_2:\mathbb{C}^m\to\mathbb{P}^1$ be two non-constant meromorphic functions. Then for $r>0$ we have 
\[\parallel\;N(r,0;f_1f_2)-N(r,f_1f_2)=N(r,0;f_1)+N(r,0;f_2)-N(r,f_1)-N(r,f_2).\]
\end{lem}

\subsection {{\bf Proof of Theorem \ref{t2.2}}} 

\begin{proof}
	We start from the given differential equation  
	\begin{equation}\label{rbc1}
		f^n \partial_u(f)\, g^n \partial_u(g) \equiv 1.
	\end{equation}
	
	\medskip
	\noindent
	\textbf{Step 1. Preliminary setup.}
	Define 
	\[
	I = I_f \cup I_g, \qquad 
	S = \bigcup_{a \in \{0, \infty\}} \big( \operatorname{supp}\mu_f^a \big)_s \cup \big( \operatorname{supp}\mu_g^a \big)_s,
	\]
	where $A_s$ denotes the set of singular points of the analytic set $A$.  
	Note that $\dim (I \cup S) \leq m-2$. Choose $z_0 \in \mathbb{C}^m - (I \cup S)$.
	
	\smallskip
	If possible, suppose $\mu_f^0(z_0) = l > 0$. Then by (\ref{rbc1}), we have $\mu_g^{\infty}(z_0) = k > 0$.  
	Since $z_0 \notin S$, there exists a holomorphic coordinate system
	\[
	(U; \varphi_1, \ldots, \varphi_m)
	\]
	around $z_0$ in $\mathbb{C}^m - (I \cup S)$ such that  
	\[
	U \cap \operatorname{supp}\mu_g^{\infty} = \{ z \in U \mid \varphi_1(z) = 0 \},
	\qquad
	\varphi_j(z_0) = 0, \quad j = 1, \ldots, m
	\]
	(see Lemma 2.3 in \cite{FL1}). So biholomorphic coordinate transformation $z_j=z_j(\varphi_1, \ldots, \varphi_m)$, $j=1, \ldots, m$ near $0$ exists such that $z_0=z(0)=(z_1(0), \ldots, z_m(0))$. So we can write
	\[
	f = \varphi_1^{l}\hat{f}(\varphi_1, \ldots, \varphi_m),
	\qquad
	g = \varphi_1^{-k}\hat{g}(\varphi_1, \ldots, \varphi_m),
	\]
	where $\hat{f}$ and $\hat{g}$ are holomorphic and non-vanishing along $\operatorname{supp}\mu_g^{\infty}$.
	
	\medskip
	\noindent
	\textbf{Step 2. Computing directional derivatives.}
	For $i \in \{1, \ldots, m\}$, we have
	\[
	\frac{\partial f}{\partial z_i} 
	= \sum_{j=1}^m \frac{\partial f}{\partial \varphi_j}\frac{\partial \varphi_j}{\partial z_i}
	= l \varphi_1^{l-1}\hat{f}\frac{\partial \varphi_1}{\partial z_i}
	+ \varphi_1^l \sum_{j=1}^m \frac{\partial \hat{f}}{\partial \varphi_j}\frac{\partial \varphi_j}{\partial z_i}.
	\]
	Hence,
	\begin{equation}\label{xx0}
		\partial_u(f) = l \varphi_1^{l-1}\hat{f}\partial_u(\varphi_1)
		+ \varphi_1^l \sum_{j=1}^m \frac{\partial \hat{f}}{\partial \varphi_j}\partial_u(\varphi_j).
	\end{equation}
	Similarly,
	\begin{equation}\label{rbc1a}
		\partial_u(g) = -\frac{k}{\varphi_1^{k+1}}\hat{g}\partial_u(\varphi_1)
		+ \frac{1}{\varphi_1^k}\sum_{j=1}^m \frac{\partial \hat{g}}{\partial \varphi_j}\partial_u(\varphi_j).
	\end{equation}
	
	\medskip
	\noindent
	\textbf{Step 3. Expanding the main identity.}
	We obtain
	\begin{align}
		f^n\partial_u(f) &= l \varphi_1^{(n+1)l-1}\hat{f}^{n+1}\partial_u(\varphi_1)
		+ \varphi_1^{(n+1)l}\hat{f}^n \sum_{j=1}^m \frac{\partial \hat{f}}{\partial \varphi_j}\partial_u(\varphi_j), \label{xx1} \\
		g^n\partial_u(g) &= -\frac{k}{\varphi_1^{(n+1)k+1}}\hat{g}^{n+1}\partial_u(\varphi_1)
		+ \frac{1}{\varphi_1^{(n+1)k}}\hat{g}^n \sum_{j=1}^m \frac{\partial \hat{g}}{\partial \varphi_j}\partial_u(\varphi_j). \label{xx2}
	\end{align}
	From (\ref{rbc1}) and (\ref{xx1})-(\ref{xx2}), it follows that $(n+1)(l - k) = 2$, a contradiction.  
	Hence $l = \mu_f^0(z_0) = 0$, so $f \neq 0$. Similarly, $g \neq 0$. Again from (\ref{rbc1a}), we get 
\[\parallel N(r,\partial_u(g))\leq N(r,g)+\ol N(r,g).\]

Similarly 
\[\parallel N(r,\partial_u(f))\leq N(r,f)+\ol N(r,f).\]
	
	\medskip
	\noindent
	\textbf{Step 4. Setting up the auxiliary function.}
	Let
	\[
	h = \frac{1}{fg}.
	\]
	Then $h$ is an entire function. We consider two cases.
	
	\bigskip
	\noindent
	\textbf{Case 1:} $h$ is non-constant.  
	\smallskip
	Since $f\neq 0$, from (\ref{rbc1}), we see that $\mu^{\infty}_{g^n\partial_ug}\leq \mu_{\partial_u f}^0$ holds over $\text{supp}\;\mu^{\infty}_f$ generically and so (\ref{xx2}) gives 
$(n+1)\mu_g^{\infty}(z_0)+\mu_{g,1}^{\infty}(z_0)\leq \mu_{\partial_u f}^0(z_0)$.
Therefore 

	\begin{equation}\label{rbc3}
		\parallel\;(n+1) N(r, g)+\overline{N}(r, g)\leq N(r,0;\partial_u(f))
	\end{equation}
	
	Now by Lemma \ref{l7a}, we have
\bea\label{rbc4}
\parallel\;N\left(r,f/\partial_u(f)\right)-N\left(r,\partial_u(f)/f\right)&=&N(r,f)+N(r,0;\partial_u(f))-N(r,\partial_u(f))-N(r,0;f)\nonumber\\
&=&N(r,0;\partial_u(f))+N(r, f)-N(r,\partial_u(f)).\eea

Using the first main theorem, we get 
\[\parallel\;N\left(r,f/\partial_u(f)\right)-N\left(r,\partial_u(f)/f\right)=m\left(r,\partial_u(f)/f\right)-m\left(r,f/\partial_u(f)\right)+O(1)\]
and so from (\ref{rbc4}), we have 
\bea\label{rbc5} \parallel\;N(r,0;\partial_u(f))=N(r,\partial_u(f))-N(r, f)+m\left(r,\partial_u(f)/f\right)-m\left(r,f/\partial_u(f)\right)+O(1).\eea

Again from (\ref{rbc1}), we have $\partial_ug / g=\left(f / \partial_u(f)\right)h^{n+1}$ and so 
\[\parallel m(r,\partial_u(g)/g)\leq m(r,f/\partial_u(f))+(n+1)h,\]
i.e.,
\bea\label{rbc6} \parallel\;m\left(r,f/\partial_u(f)\right) \geq m\left(r,\partial_u(g)/g\right)-(n+1) m(r,h)-O(1).\eea

Note that $\parallel N(r,\partial_u(f))\leq N(r,f)+\ol N(r,f)$
and so from (\ref{rbc3}) and (\ref{rbc5})-(\ref{rbc6}), we get
\bea\label{rbc8}\parallel\;(n+1) N(r, g)+\overline{N}(r, g)&\leq& \ol N(r, f)+m\left(r,\partial_u(f)/f\right)-m\left(r,\partial_u(g)/g\right)\\&&+(n+1) m(r,h)+O(1).\nonumber\eea

Similarly, we have 
\bea\label{rbc9}\parallel\;(n+1) N(r, f)+\overline{N}(r, f)&\leq& \ol N(r, g)+m\left(r,\partial_u(g)/g\right)-m\left(r,\partial_u(f)/f\right)\\&&+(n+1) m(r,h)+O(1).\nonumber\eea

Now adding (\ref{rbc8}) and (\ref{rbc9}), we get
\bea\label{rbc10} \parallel\;N(r,f)+N(r,g)\leq 2m(r,h)+O(1).\eea

Since $f=1/gh$, we get 
\[\partial_u(f)=-\big(g\partial_u(h)+h\partial_u(g)\big)/g^2h^2\]
and so from (\ref{rbc1}), we get
\bea\label{rbc11}\xi^2=-h^{n+1}\left(\partial_u(h)/h\right)^2/4,\eea
where
\bea\label{rbc10a}\xi=\partial_u(g)/g+(\partial_u(h)/h)/2.\eea

\smallskip
First suppose $\xi \equiv 0$. Then from (\ref{rbc11}), we get $h^{n+1}=\left(\partial_u(h)/h\right)^2/4$.
Now by Lemma \ref{l2}, we get $\parallel (n+1)T(r, h)=(n+1)\;m(r, h)=o(T(r,h))$, which is absurd. 

\smallskip
Next suppose $\xi \not \equiv 0$. Clearly (\ref{rbc11}) yields
\bea\label{rbc12a}2\xi^2 (\partial_u(\xi)/\xi)=2\xi \partial_u(\xi)=(1/2)(\partial_u(h)/h)\partial_u\left(\partial_u(h)/h\right)-(n+1) h^n\partial_u(h)\eea
and so from (\ref{rbc11}), we get
\bea\label{rbc12} h^{n+1}H_1=\partial_u(h)\left(\partial_u\left(\partial_u(h)/h\right)-\partial_u(h)/h\right)/2h,\eea
where 
\[H_1=(n+1)\partial_u(h)/h-(2\partial_u(\xi)/\xi).\]

	Further manipulations involving $\xi$ defined by
	\[
	\xi = \frac{\partial_u(g)}{g} + \frac{1}{2}\frac{\partial_u(h)}{h}
	\]
	lead to the key relation
	\begin{equation}\label{rbc11}
		\xi^2 = -\frac{h^{n+1}}{4}\left(\frac{\partial_u(h)}{h}\right)^2.
	\end{equation}
	
	Now we consider following two sub-cases.
	
	\medskip
	\noindent
	\textbf{Sub-case 1.1:} $H_1 \equiv 0$.  Then $(n+1)\partial_u(h)/h-(2\partial_u(\xi)/\xi)\equiv 0$,
i.e, 
\bea\label{rbc12aa} (n+1)\partial_u(h)h\equiv 2\partial_u(\xi)/\xi.\eea

If $\dim(\mathbb{C}^m)=1$, then from (\ref{rbc12aa}), we get $\xi^2\equiv c_1h^{n+1}$ and so (\ref{rbc11}) gives 
\[(c_1+1)h^{n+1}\equiv (1/4)(h^{(1)}/h)^2.\]

Since $h^{(1)}\not\equiv 0$, $c_1\neq -1$ and so by Lemma \ref{l2}, we get $\parallel (n+1)T(r,h)=(n+1)m(r,h)=o(T(r,h))$,
which is absurd. Hence $\dim(\mathbb{C}^m)\geq 2$. Now from (\ref{rbc12}) and (\ref{rbc12aa}), we have 
\[\partial_u\left(\partial_u(h)/h\right)\equiv \partial_u(h)/h.\]

Again from (\ref{rbc12a}), we get
\beas (n+1)\partial_u(h)/h \xi^2=(1/2) \left(\partial_u(h)/h\right)^2-(n+1) h^{n+1}\partial_u(h)/h\eeas
and so from (\ref{rbc11}), we have
\bea\label{rbc12b} ((n+1)/4)\left(\partial_u(h)/h\right)^3=\left(\partial_u(h)/h\right)^2/2.\eea

\smallskip
First suppose $\partial_u(h)\not\equiv 0$. Then (\ref{rbc12b}) gives $\partial_u(h)/h=2/(n+1)$. Since 
$\partial_u\left(\partial_u(h)/h\right)\equiv \partial_u(h)/h$, we get a contradiction.

\smallskip
Next suppose $\partial_u(h)\equiv 0$. Clearly from (\ref{rbc11}) and (\ref{rbc10a}), we have respectively $\xi^2=-h^{n+1}$ and $\xi=\partial_u(g)/g$ and so $(\partial_u(g)/g)^2=-h^{n+1}$.  
Since $h$ is an entire function and $g\neq 0$, we see that $g$ is also an entire function. 
Again $h=1/fg$ gives 
\[\partial_u(h)/h=-(\partial_u(f)/f)-(\partial_u(g)/g)\]
and so $\partial_u(f)/f=-\partial_u(g)/g$.
This shows that $f$ is also an entire function. Let us take $f=\exp(\alpha)$ and $g=\exp(\beta)$, where $\alpha$ and $\beta$ are non-constant entire functions in $\mathbb{C}^m$. Clearly $\partial_u(\alpha)\equiv-\partial_u(\beta)$. Since 
$\left(\partial_u(g)/g\right)^2=-h^{n+1}$, we have 
$(\partial_u(\beta))^2\exp(-(n+1)(\alpha+\beta))=-1$.
Consequently 
\[2\partial_u^2(\beta)-(n+1)\partial_u(\beta)(\partial_u( \alpha)+\partial_u(\beta))\equiv 0\]
and so $\partial_u^2(\beta)\equiv 0$. Finally, we have 
\[f=\exp(\alpha)\;\;\text{and}\;\;g=\exp(\beta),\]
where $\alpha$ and $\beta$ are non-constant entire functions in $\mathbb{C}^m$ such that $\partial_u(\alpha)=c$, $\partial_u(\beta)=-c$ and $c^2e^{-(n+1)(\alpha+\beta)}=-1$.

\medskip
\noindent
\textbf{Sub-case 1.2:} $H_1 \not\equiv 0$. Since $\partial_u(h)/h=-(\partial_u(f)/f)-(\partial_u(g)/g)$, from (\ref{rbc10a}), we get 
\[\xi=\left(\partial_u(g)/g-\partial_u(f)/f\right)/2\;\;\text{and so}\;\; 
\partial_u (\xi)=\left(\partial_u\left(\partial_u(g)/g\right)-\partial_u\left(\partial_u(f)/f\right)\right)/2.\] 

Therefore
\[H_1=-(n+1)\left(\partial_u(f)/f+\partial_u(g)/g\right)-\frac{\partial_u\left(\partial_u(g)/g\right)-\partial_u\left(\partial_u(f)/f\right)}{\xi}.\]

Since $f\neq 0$ and $g\neq 0$, it is easy to verify that $\mu^{\infty}_{H_1}\leq \mu^{\infty}_{f,1}+\mu^{\infty}_{g,1}+\mu^0_{\xi,1}$
holds over $\text{supp}\;\mu^{\infty}_{H_1}$ generically and so
\bea\label{rbc15} \parallel\;N(r,H_1)\leq \ol N(r,f)+\ol N(r,g)+\ol N(r,0;\xi).\eea

On the other hand using the first main theorem and Lemma \ref{l2} to (\ref{rbc11}), we have 
\bea\label{rbc16} \parallel\;2T(r,\xi)\leq 2T\left(r,\partial_u(h)/h\right)+(n+1)T(r,h)&\leq& 2N\left(r,\partial_u(h)/h\right)+(n+1)T(r,h)\\&\leq& 2\ol N(r,0;h)+(n+1)T(r,h)\nonumber\\&\leq& (n+3)T(r,h)+o(T(r,h)).\nonumber\eea

Again using the first main theorem, Lemma \ref{l2}, (\ref{rbc10}) and (\ref{rbc12})-(\ref{rbc16}), we get
\beas
\parallel\;(n+1) T(r, h)&\leq & m(r, h^{n+1}H_1)+m(r,0;H_1)+O(1)\nonumber \\
&\leq & m\left(r,(1/2)(\partial_u(h)/h)\left(\partial_u\left(\partial_u(h)/h)\right)-\partial_u(h)/h\right)\right)+T(r,H_1)\nonumber \\
&\leq & N(r,H_1)+o(T(r, h))+o(T(r, \xi))\nonumber\\
&\leq& \ol N(r,f)+\ol N(r,g)+\ol N(r,0;\xi)+o(T(r, h))+o(T(r,\xi))\nonumber\\
&\leq& ((n+7)/2)T(r,h)+o(T(r, h)),\nonumber\eeas
i.e., $\parallel\;(1/2)(n-5)T(r,h)\leq o(T(r,h))$, which is impossible.

\bigskip
	\noindent
	\textbf{Case 2:} $h$ is constant. Since $f\neq 0$ and $g\neq 0$,  both $f$ and $g$ are non-constant entire functions. Note that $f=1/gh$ and so $\partial_u(f)=-\partial_u(g)/g^2h$. Consequently  
	\[
	\frac{\partial_u(f)}{f} = -\frac{\partial_u(g)}{g},
	\]
	and so from (\ref{rbc1}) we get
	\[
	\partial_u(g) = cg, \quad \partial_u(f) = -cf,
	\]
	where $c = i h^{(n+1)/2}$.
	Hence,
	\[
	f = e^{\alpha}, \quad g = e^{\beta},
	\]
	where $\alpha, \beta$ are non-constant entire functions in $\mathbb{C}^m$ satisfying
	\[
	\partial_u(\alpha) = c, \quad
	\partial_u(\beta) = c, \quad
	c^2 e^{-(n+1)(\alpha+\beta)} = -1.
	\]
	
	\medskip
	\noindent
	\textbf{Conclusion.}
	In all cases, the only admissible non-trivial solutions are exponential-type functions:
	\[
	f = e^{\alpha}, \quad g = e^{\beta},
	\]
	with $\partial_u(\alpha)$ and $\partial_u(\beta)$ satisfying the above relations. Hence the proof.
\end{proof}

\subsection {{\bf Proof of Theorem \ref{t2.1}}} 
The proof of Theorem \ref{t2.1} follows directly from the proof of Theorem \ref{t2.2}. So we omit the detail.

\subsection*{\textbf{Proof of Corollary \ref{c2.2}}}
\begin{proof}
	By Theorem \ref{t2.2}, we have 
	\[
	f = \exp(\alpha), \qquad g = \exp(\beta),
	\]
	where $\alpha$ and $\beta$ are non-constant entire functions on $\mathbb{C}^m$ satisfying
	\[
	\partial_j(\alpha) = c, \qquad \partial_j(\beta) = c \quad \text{for all } j \in \mathbb{Z}[1,m],
	\]
	and
	\[
	c^2 e^{-(n+1)(\alpha+\beta)} = -1.
	\]
	
	\bigskip
	\noindent
	\textbf{Step 1. Showing that $\alpha$ is a polynomial.}
	From the given condition, we have
	\begin{equation}\label{n.1}
		\frac{\partial^2 \alpha(z)}{\partial z_i^2} \equiv 0, \qquad i = 1,2,\ldots,m.
	\end{equation}
	We will prove that $\alpha(z)$ is a polynomial in $\mathbb{C}^m$ by induction on $m$.
	
	\medskip
	\noindent
	\textbf{Base Case:} $\boldsymbol{m = 1.}$  
	If $\dim(\mathbb{C}^m) = 1$, then from (\ref{n.1}) it follows that $\alpha(z)$ is a polynomial in one variable.
	
	\medskip
	\noindent
	\textbf{Induction Step:} Let us suppose that $\dim(\mathbb{C}^m)=2$. Since 
\[\frac{\partial^2\alpha(z_1,z_2)}{\partial z_2^2}\equiv 0,\]
on integration, we have 
\bea\label {n.2} \alpha(z_1,z_2)=\phi_1(z_1)z_2+\phi_2(z_1),\eea
where $\phi_i(z_1)$'s are entire functions in $\mathbb{C}$ in the variable $z_1$. Note that $\frac{\partial^2\alpha(z_1,z_2)}{\partial z_1^2}\equiv 0$ and so (\ref{n.2}) gives $\phi_1^{(2)}(z_1)z_2+\phi_2^{(2)}(z_1)\equiv 0$,
which shows that $\phi_1^{(2)}(z_1)\equiv 0$ and $\phi_2^{(2)}(z_1)\equiv 0$. Therefore on integration, we get
$\phi_1(z_1)=c_1z_1+c_2$ and $\phi_2(z_1)=d_1z_1+d_2$, where $c_1,c_2,d_1$ and $d_2$ are constants in $\mathbb{C}$. Therefore from (\ref{n.2}), we get $\alpha(z_1,z_2)=(c_1z_1+c_2)z_2+(d_1z_1+d_2)$,
which shows that $\alpha(z_1,z_2)$ is a polynomial in $\mathbb{C}^2$. Now we fix $\dim(\mathbb{C}^m)\geq 2$ and assume that $\alpha(z)$ is a polynomial for variables of number at most $m-1$. Since $\frac{\partial^2\alpha(z)}{\partial z_m^2}\equiv 0$, we have 
\bea\label{n.3} \alpha(z_1,z_2,\ldots,z_m)= A(z_1,z_2,\ldots,z_{m-1})z_m+B(z_1,z_2,\ldots,z_{m-1}).\eea

Now using (\ref{n.1}) to (\ref{n.3}), we get
\bea\label{n.4} \frac{\partial^2 A(z_1,z_2,\ldots,z_{m-1})}{\partial z_i^2}z_m+\frac{\partial^2 B(z_1,z_2,\ldots,z_{m-1})}{\partial z_i^2}\equiv 0\eea
for $i=1,2,\ldots,m-1$. Therefore (\ref{n.4}) yields
\bea\label{n.5} \frac{\partial^2 A(z_1,z_2,\ldots,z_{m-1})}{\partial z_i^2}\equiv 0\;\;\text{and}\;\;\frac{\partial^2 B(z_1,z_2,\ldots,z_{m-1})}{\partial z_i^2}\equiv 0\eea
for $i=1,2,\ldots,m-1$. Then by the induction assumptions we find from (\ref{n.5}) that both $A(z_1,z_2,\ldots,z_{m-1})$ and $B(z_1,z_2,\ldots,z_{m-1})$ are polynomials in the variables $z_1,z_2,\ldots,z_m$. Therefore from (\ref{n.3}), we get that $\alpha(z_1,z_2,\ldots,z_m)$ is a polynomial in $z_1,z_2,\ldots,z_m$. Since $\partial_j(\alpha)=c$ for all $j\in\mathbb{Z}[1,m]$, we may assume that $\alpha(z)=c(z_1+z_2+\ldots+z_m)+d_1$.
	
	\bigskip
	\noindent
	\textbf{Step 2. Determining the form of $\beta$.}
	By the same argument, $\beta(z)$ is also a polynomial in $(z_1, z_2, \ldots, z_m)$.  
	Using $\partial_j(\beta) = c$ and the relation from equation (\ref{2.1}), we may take
	\[
	\beta(z) = -c(z_1 + z_2 + \cdots + z_m) + d_2,
	\]
	where $d_2$ is a constant.
	
	\bigskip
	\noindent
	\textbf{Step 3. Final form of $f$ and $g$.}
	Substituting $\alpha$ and $\beta$ into $f = e^{\alpha}$ and $g = e^{\beta}$, we obtain
	\[
	f(z) = c_1 \exp\big(c(z_1 + z_2 + \cdots + z_m)\big),
	\qquad
	g(z) = c_2 \exp\big(-c(z_1 + z_2 + \cdots + z_m)\big),
	\]
	where $c, c_1, c_2$ are nonzero constants satisfying
	\[
	c^2 (c_1 c_2)^{n+1} = -1.
	\]
	\medskip
	Hence the proof.
\end{proof}

\section {{\bf Uniqueness problems}}

In this section, our main objective is to consider the uniqueness problem for meromorphic functions in $\mathbb{C}^m$ concerning nonlinear differential polynomials generated by them sharing one value. Also our result ensures that Theorem C works in several complex
variables. Now we state our main results.

\begin{theo}\label{t3.1} Let $f:\mathbb{C}^m\to \mathbb{P}^1$ and $g:\mathbb{C}^m\to \mathbb{P}^1$ be two non-constant meromorphic functions and let $n\geq 11$ be an integer. 
	If $f^n\partial_u(f)$ and $g^n\partial_u(g)$ share $a$ CM, where $a\in\mathbb{C}\backslash \{0\}$, then one of the following cases holds:
	\begin{enumerate}
		\item[(i)] $f\equiv dg$ for some $(n+1)$-th root of unity  $d$;
		\item[(ii)] $f=\exp(\alpha)$ and $g=\exp(\beta)$, where $c\in\mathbb{C}\backslash \{0\}$, $\alpha$ and $\beta$ are non-constant entire functions in $\mathbb{C}^m$ such that $\partial_u(\alpha)=c$, $\partial_u(\beta)=-c$ and $c^2e^{(n+1)(\alpha+\beta)}=-a^2$. 
	\end{enumerate} 
\end{theo}

\begin{theo}\label{t3.2} Let $f:\mathbb{C}^m\to \mathbb{P}^1$ and $g:\mathbb{C}^m\to \mathbb{P}^1$ be two non-constant meromorphic functions and let $n\geq 11$ be an integer. If $f^n\partial_{j}(f)$ and $g^n\partial_{j}(g)$ share $a$ CM, where $j\in\mathbb{Z}[1,m]$ and $a\in\mathbb{C}\backslash \{0\}$, then one of the following cases holds:
	\begin{enumerate}
		\item[(i)] $f\equiv dg$ for some $(n+1)$-th root of unity $d$;
		\item[(ii)] $f=\exp(\alpha)$ and $g=\exp(\beta)$, where $c\in\mathbb{C}\backslash \{0\}$, $\alpha$ and $\beta$ are non-constant entire functions in $\mathbb{C}^m$ such that $\partial_j(\alpha)=c$, $\partial_j(\beta)=-c$ and $c^2e^{(n+1)(\alpha+\beta)}=-a^2$. 
	\end{enumerate} 
\end{theo}

For non-constant entire functions, we get the following corollary from Theorem \ref{t3.1}.
\begin{cor}\label{c3.1} Let $f:\mathbb{C}^m\to \mathbb{C}$ and $g:\mathbb{C}^m\to \mathbb{C}$ be two non-constant entire functions such that $\partial_u(f)\not\equiv 0$ and $\partial_u (g)\not\equiv 0$ and let $n\geq 7$ be an integer. If $f^n\partial_u(f)$ and $g^n\partial_u(g)$ share $a$ CM, where $a\in\mathbb{C}\backslash \{0\}$, then conclusions of Theorem \ref{t3.1} hold.
\end{cor}

Following example ensures the necessity of the condition ``$a\in\mathbb{C}\backslash \{0\}$'' in Corollary \ref{c3.1}.
\begin{exm} Let $f(z)=\exp(\exp(z_1+\ldots+z_m))$ and $g(z)=\exp(z_1+\ldots+z_m)$. Clearly $f^n\partial_u(f)$ and $g^n\partial_u(g)$ share $0$ CM, but $f$ and $g$ do not satisfy the conclusion of Corollary \ref{c3.1}.
\end{exm}

\subsection {{\bf Auxiliary lemmas}}
In order to prove Theorems \ref{t3.1}-\ref{t3.2}, we require following lemmas.

\begin{lem}\label{l1}\cite[Corollary 1.40]{HLY1} Let $f_0, \ldots, f_n$ be linearly independent meromorphic functions in $\mathbb{C}^m$. Write $f = (f_0, \ldots, f_n)$. Then there are multi-index $\nu_i \in \mathbb{Z}_+^m$ with $0 < |\nu_i| \leq i$ $(i = 1, \ldots, n)$ such that $f,\partial^{\nu_1}(f), \ldots,\partial^{\nu_n}(f)$ are linearly independent over $\mathbb{C}^m$ and there exists partition 
	\[\lbrace \nu_1,\ldots,\nu_n\rbrace=\sideset{}{_{i=1}^s}{\bigcup} I_i\;\;(1\leq s \leq n)\]
	such that $I_k \subset\{\nu \in \mathbb{Z}_+^m \mid |\nu| = k\}$, $k = 1, \ldots, s$ and when $1 \leq k < s$ each element in 
	\[\lbrace \partial^\nu f \mid \nu \in \mathbb{Z}_+^m, |\nu| = k, \nu \notin I_k\rbrace\]
	can be expressed as a linear combination of the family 
	\[\lbrace f, \partial^\nu(f) \mid \nu \in \sideset{}{_{i=1}^k}{\bigcup} I_i\rbrace.\]
	
\end{lem}

\begin{rem}\label{r3.1} By using Lemma 1.41 \cite{HLY1}, we make a remark on Lemma \ref{l1}. If $\#(I_s)=1$ in Lemma \ref{l1}, say, $\nu_n\in I_s$, then Lemma 1.41 \cite{HLY1} shows that we may choose $\nu_n$ such that $\nu_n=\nu+\imath_j$ for some $\nu\in I_{s-1}$, $j\in \mathbb{Z}[1, m]$, where $\imath_j=(0,\ldots,0,1,0,\ldots,0)\in\mathbb{Z}^m_{+}$ in which $1$ is $j$-th component of $\imath_j$.
\end{rem}

\begin{lem}\label{l3} \cite[Lemma 1.2]{HY1} Let $f:\mathbb{C}^m\to\mathbb{P}^1$ be a non-constant meromorphic function and let $a_1,a_2,\ldots,a_q$ be different points in $\mathbb{P}^1$. Then
	\[\parallel (q-2)T(r,f)\leq \sideset{}{_{j=1}^{q}}{\sum} \ol N(r,a_j;f)+O(\log (rT(r,f))).\]
\end{lem}

\begin{lem}\label{l4} \cite[Theorem 1.26]{HLY1} Let $f:\mathbb{C}^m\to\mathbb{P}^1$ be a non-constant meromorphic function. If
	$R(z, w)=\frac{A(z, w)}{B(z, w)}$, then
	\[\parallel\;T(r, R_f)=\max \{p, q\} T(r, f)+O\big(\sideset{}{_{j=0}^{p}}{\sum}T(r, a_j)+\sideset{}{_{j=0}^{q}}{\sum}T(r, b_j)\big),\]
	where $R_f(z)=R(z, f(z))$ and two coprime polynomials $A(z, w)$ and $B(z,w)$ are given
	respectively $A(z,w)=\sum_{j=0}^p a_j(z)w^j$ and $B(z,w)=\sum_{j=0}^q b_j(z)w^j$.
\end{lem}

\begin{lem}\label{l5} \label{l6}\cite[Theorem 1.101]{HLY1} Suppose that $f_1, f_2,\ldots, f_n$ are linearly independent meromorphic
	functions in $\mathbb{C}^m$ such that $\sum_{i=1}^nf_i=1$. Then for $1\leq j\leq n$ and $R>\rho>r>r_0$, 
	\beas \parallel\;T(r,f_j)&\leq & N(r,f_j)+\sideset{}{_{j=1}^{n}}{\sum}\left\lbrace  N(r,0;f_j)-N(r,f_j)\right\rbrace+N(r,\mathbf{W})\\
	&&-N(r,0;\mathbf{W})+l_1 \log\left\lbrace (\rho/r)^{2m-1}T(R)/(\rho-r)\right\rbrace+O(1),\eeas
	where $T(R)=\max\limits_{1\leq j\leq n} \{T(r,f_j)\}$ and $\mathbf{W}=\mathbf{W}_{\nu_1\ldots\nu_{n-1}}(f_1,f_2,\ldots,f_n)\not\equiv 0$ is the Wronskian determinant and $n-1\leq l_1=|\nu_1|+\ldots+|\nu_{n-1}|\leq n(n-1)/2$.
\end{lem}

\begin{lem}\label{l6} Let $f:\mathbb{C}^m\to \mathbb{P}^1$ and $g:\mathbb{C}^m\to\mathbb{P}^1$ be two non-constant meromorphic functions. If $f$ and $g$ share $1$ CM, then one of the following three cases holds:
	\begin{enumerate}
		\item[(i)] $\parallel\; T(r,f)\leq N_2(r,f)+N_2(r,g)+N_2(r,0;f)+N_2(r,0;g)+o(T(r))$,\\
		where $T(r)=\max\{T(r,f),T(r,g)\}$ and the same inequality holding for $T(r,g)$,
		
		\smallskip
		\item[(ii)] $f\equiv g$,
		
		\smallskip
		\item[(iii)] $fg\equiv 1$.
	\end{enumerate}
\end{lem}

\begin{proof} We consider following two cases.
	
	\smallskip
	{\bf Case 1.} Let $f-1$ and $g-1$ be linearly dependent. Since $f$ and $g$ are non-constant, there exist $c_1\neq 0$ and $c_2\neq 0$ such that $c_1(f-1)+c_2(g-1)=0$. Then
	\bea\label{rbb1} g-1=d(f-1),\eea
	where $d=-c_1/c_2$. If $d=1$, then $f\equiv g$ and so we get $(ii)$. Next let $d\neq 1$. Then (\ref{rbb1}) gives 
	$\parallel \ol N(r,0;g)=\ol N(r,(d-1)/d;f)$ and so by Lemma \ref{l3}, we get
	\beas \parallel T(r,f)&\leq& \ol N(r,0;f)+\ol N(r,f)+\ol N(r,(d-1)/d;f)+o(T(r,f))\\&\leq&
	N_2(r,f)+N_2(r,g)+N_2(r,0;f)+N_2(r,0;g)+o(T(r,f)).
	\eeas
	
	Therefore we obtain $(i)$.

	\smallskip
	{\bf Case 2.} Let $f-1$ and $g-1$ be linearly independent. Suppose 
	\bea\label{rbb2} F_1=f,\;\;F_2=(f-1)/(g-1)\;\;\text{and}\;\;F_3=-gF_2.\eea
	
	Now we consider following two sub-cases.
	
	\smallskip
	{\bf Sub-case 2.1.} Let $F_1$, $F_2$ and $F_3$ be linearly independent.
	Clearly (\ref{rbb2}) gives $F_1+F_2+F_3=1$
	and so by Lemma \ref{l5}, we get 
	\bea\label{rbb3} \parallel\;T(r,F)&\leq &N(r,F_1) + \sideset{}{_{k=1}^{3}}{\sum} \Big\lbrace N(r,0;F_k)-N(r,F_k)\Big\rbrace+ N(r,\mathbf{W})\nonumber\\&&
	- N(r,0;\mathbf{W})+l_1\log \left\lbrace (\rho/r)^{2m-1}T_1(R)/(\rho-r)\right\rbrace+ O(1),\eea
	where $T_1(r)=\max\{T(r,F_1),T(r,F_2),T(r,F_3)\}$, $2\leq l=|\nu_1|+|\nu_2|\leq 3$ 
	and
	\beas \mathbf{W}= \begin{vmatrix} F_1&F_2&F_3\\
		\partial^{\nu_1}F_1&\partial^{\nu_1}F_2&\partial^{\nu_1}F_3\\
		\partial^{\nu_2}F_1&\partial^{\nu_2}F_2&\partial^{\nu_2}F_3
	\end{vmatrix}.
	\eeas
	
	Since $F_1+F_2+F_3=1$, we have
	\bea\label{rbb4} \mathbf{W}=\partial^{\nu_1}F_1 \partial^{\nu_2}F_2 - \partial^{\nu_1}F_2 \partial^{\nu_2}F_1,\eea 
	\bea\label{rbb5} \mathbf{W}=\partial^{\nu_1}F_3 \partial^{\nu_2}F_1 - \partial^{\nu_1}F_1 \partial^{\nu_2}F_3. \eea 
	
	\smallskip
	We now prove the following inequality
	\bea\label{rbb8}\mu:=\sideset{}{_{k=1}^{3}}{\sum}\mu_{F_k}^0-\mu_{F_2}^{\infty}-\mu_{F_3}^\infty+\mu_{\mathbf{W}}^\infty-\mu_{\mathbf{W}}^0
	\leq \mu_{f,2}^0+\mu_{g,2}^0+\mu_{f,2}^\infty+\mu_{g,2}^\infty:=\nu.\eea
	
	\smallskip
	Note that generically, we have 
	\bea\label{rbb8a}\mu_{F_2}^0=\max\{\mu_{g}^{\infty}-\mu^{\infty}_f,0\}\;\text{and}\;\mu_{F_2}^{\infty}=\max\{\mu_{f}^{\infty}-\mu^{\infty}_g,0\}.\eea
	
	Define 
	\[I=I_f\cup I_g\cup I_{F_2}\cup I_{F_3}\;\;\text{and}\;\;
	S=\sideset{}{_{a\in\{0,1,\infty\}}}{\bigcup}(\operatorname{supp}\mu_f^a)_s\cup(\operatorname{supp}\mu_g^a)_s,\]
	where $A_s$ denotes the set of singular points of the analytic set $A$. It suffices to prove (\ref{rbb8}) 
	on $\mathbb{C}^m-(I\cup S)$, since $\dim (I\cup S)\leq m-2$. Take $z_0\in \mathbb{C}^m-(I\cup S)$.
	
	\smallskip
	First we assume $\mu_f^1(z_0)>0$. Since $\operatorname{supp} \mu_f^1=\operatorname{supp}\mu_g^1$, we have $\mu_g^1(z_0)>0$.
	By the given condition we get $\mu^1_f(z_0)=\mu^1_g(z_0)$. Clearly $\mu(z_0)\leq 0$ and so the inequality (\ref{rbb8}) holds.
	
	\smallskip
	Next we assume $\mu_f^1(z_0)=0$. Then $\mu_g^1(z_0)=0$. 
	Now we will prove the inequality (\ref{rbb8}) by distinguishing following four sub-cases.

	\smallskip
	{\bf Sub-case 2.1.1.} Let $\mu_f^\infty(z_0)=0$ and $\mu_g^\infty(z_0)=0$. In this case from (\ref{rbb8a}), we have \[\mu(z_0)=\mu_f^0(z_0)+\mu_g^0(z_0)-\mu_{\mathbf{W}}^0(z_0).\]
	
	Also from (\ref{rbb5}), we find
	\beas \quad \mu_{\mathbf{W}}^0(z_0)\geq \mu_{F_1}^0(z_0)-\mu_{F_1,2}^0(z_0)+\mu_{F_3}^0(z_0)-\mu_{F_3,2}^0(z_0)
	=\mu_f^0(z_0)-\mu_{f,2}^0(z_0)+\mu_g^0(z_0)-\mu_{g,2}^0(z_0).\eeas
	
	Therefore $\mu(z_0)\leq \mu_{f,2}^0(z_0)+\mu_{g,2}^0(z_0)=\nu(z_0)$
	and so inequality (\ref{rbb8}) holds.

	\smallskip
	{\bf Sub-case 2.1.2.} Let $\mu_f^{\infty}\left(z_0\right)=0$ and $\mu_g^{\infty}\left(z_0\right)>0$.
	In this case from (\ref{rbb8a}), we have 
	\[\mu\left(z_0\right)=\mu_f^0\left(z_0\right)+\mu_g^{\infty}\left(z_0\right)-\mu_{\mathbf{W}}^0\left(z_0\right).\]
	
	Now from (\ref{rbb4}), we get
	\beas \quad\mu_{\mathbf{W}}^0(z_0)\geq \mu_{F_1}^0(z_0)-\mu_{F_1, 2}^0(z_0)+\mu_{F_2}^0(z_0)-\mu_{F_2, 2}^0(z_0)
	=\mu_f^0(z_0)-\mu_{f, 2}^0(z_0)+\mu_g^{\infty}(z_0)-\mu_{g, 2}^{\infty}(z_0).\eeas
	
	Consequently 
	\[\mu(z_0) \leq \mu_{f, 2}^0(z_0)+\mu_{g, 2}^{\infty}(z_0)=\nu(z_0)\]
	and so the inequality (\ref{rbb8}) holds.
	
	\smallskip
	{\bf Sub-case 2.1.3.} Let $\mu_f^{\infty}\left(z_0\right)>0$ and $\mu_g^{\infty}\left(z_0\right)=0$. Clearly
	\bea\label{rbb9} \quad\mu(z_0)=\mu_{F_3}^0(z_0)-\mu_{F_2}^{\infty}(z_0)-\mu_{F_3}^{\infty}(z_0)+\mu_{\mathbf{W}}^{\infty}(z_0)-\mu_{\mathbf{W}}^0(z_0).\eea
	
	Let $l=\mu_f^{\infty}(z_0)$ and $k=\mu_f^{\infty}(z_0)-\mu_g^0(z_0)$.
	
	If $k=0$, i.e., if $\mu_g^0(z_0)=\mu_f^{\infty}(z_0)>0$, then from (\ref{rbb5}), we have $\mu_{\mathbf{W}}^{\infty}(z_0) \leq \mu_{F_1}^{\infty}(z_0)+2$ and so by (\ref{rbb9}), we get
	\[\quad\mu(z_0)\leq-\mu_{F_2}^{\infty}(z_0)+\mu_{\mathbf{W}}^{\infty}(z_0)\leq \mu_{F_1}^{\infty}(z_0)-\mu_{F_2}^{\infty}(z_0)+2=2\leq \nu(z_0).\]
	
	Consequently the inequality (\ref{rbb8}) holds.
	
	\smallskip
	Next we suppose $k\neq 0$. Since $2\leq |\nu_1|+|\nu_2|\leq 3$, for the sake of simplicity we may assume that $\partial^{\nu_1}=\partial_{z_i}$ for some $i\in\mathbb{Z}[1,m]$.
	Since $z_0 \notin S$, there is a holomorphic coordinate system $\left(U ; \varphi_1, \ldots, \varphi_m\right)$ of $z_0$ in $\mathbb{C}^m-(I \cup S)$ such that $U \cap \operatorname{supp} \mu_f^{\infty}=\left\{z \in U \mid \varphi_1(z)=0\right\}$ and $\varphi_j(z_0)=0$ for $j=1,2,\ldots,m$.
	So biholomorphic coordinate transformation $z_j=z_j(\varphi_1, \ldots, \varphi_m)$
	near $0$ exists such that $z_0=z(0)=(z_1(0), \ldots, z_m(0))$ such that
	\bea\label{rbb10}
	F_1=\varphi_1^{-l} \hat{F}_1\left(\varphi_1, \ldots, \varphi_m\right),\;\;F_2=\varphi_1^{-l} \hat{F}_2\left(\varphi_1, \ldots, \varphi_m\right)\;\text{and}\;
	F_3=\varphi_1^{-k} \hat{F}_3\left(\varphi_1, \ldots, \varphi_m\right),
	\eea
	where $\hat{F}_i\;(i=1,2,3)$ are holomorphic functions near $0$, which do not vanish along the set $\operatorname{supp} \mu_f^{\infty}$. 
	Now we consider following two sub-cases.

	\smallskip
	{\bf Sub-case 2.1.3.1.} Let $\left|\nu_2\right|=1$. Suppose $\partial^{\nu_2}=\partial_{z_j}$ for some $j \in \mathbb{Z}[1, m]-\{i\}$. Now using (\ref{rbb10}) to (\ref{rbb5}), we get 
	\bea\label{rbb11}
	\mathbf{W}&=& \big(-k \varphi_1^{-k-1} \hat{F}_3 \partial_{z_i} \varphi_1+\varphi_1^{-k} \partial_{z_i} \hat{F}_3\big)\big(-l \varphi_1^{-l-1} \hat{F}_1 \partial_{z_j} \varphi_1+\varphi_1^{-l} \partial_{z_j} \hat{F}_1\big) \\
	&& -\big(-l \varphi_1^{-l-1} \hat{F}_1 \partial_{z_i} \varphi_1+\varphi_1^{-l} \partial_{z_i} \hat{F}_1\big)\big(-k \varphi_1^{-k-1} \hat{F}_3 \partial_{z_j} \varphi_1+\varphi_1^{-k} \partial_{z_j} \hat{F}_3\big)\nonumber \\
	&=& \varphi_1^{-l-k-1}\big(l \hat{F}_1 \partial_{z_i} \varphi_1 \partial_{z_j} \hat{F}_3+k \hat{F}_3 \partial_{z_i} \hat{F}_1 \partial_{z_j} \varphi_1
	-l \hat{F}_1 \partial_{z_i} \hat{F}_3 \partial_{z_j} \varphi_1-k \hat{F}_3 \partial_{z_i} \varphi_1 \partial_{z_j} \hat{F}_1\big)\nonumber \\
	&& +\varphi_1^{-l-k}\big(\partial_{z_i} \hat{F}_3 \partial_{z_j} \hat{F}_1-\partial_{z_i} \hat{F}_1 \partial_{z_j} \hat{F}_3\big).\nonumber
	\eea
	
	\smallskip
	Let $-l-k>0$. Then from (\ref{rbb11}), we get $\mu_{\mathbf{W}}^0(z_0) \geq-l-k-1=-2 \mu_f^{\infty}(z_0)+\mu_g^0(z_0)-1$
	and so by (\ref{rbb9}), we find
	\[\mu(z_0)=\mu_g^0(z_0)-2 \mu_f^{\infty}(z_0)-\mu_{\mathbf{W}}^0(z_0) \leq 1 \leq \mu_{f, 2}^{\infty}(z_0) \leq \nu(z_0),\]
	which shows that the inequality (\ref{rbb8}) holds.\par
	
	\smallskip
	Let $-l-k \leq 0$. Then (\ref{rbb11}) gives
	$\mu_{\mathbf{W}}^{\infty}(z_0) \leq l+k+1=2 \mu_f^{\infty}(z_0)-\mu_g^0(z_0)+1$
	and so from (\ref{rbb9}), we get
	\[\mu(z_0) \leq \mu_g^0(z_0)-2 \mu_f^{\infty}(z_0)+\mu_{\mathbf{W}}^{\infty}(z_0) \leq 1 \leq \mu_{f, 2}^{\infty}(z_0) \leq \nu(z_0),\]
	which shows that the inequality (\ref{rbb8}) holds.

	\smallskip
	{\bf Sub-case 2.1.3.2.} Let $\left|\nu_2\right|=2$. Now by Remark \ref{r3.1}, there exists some $j\in \mathbb{Z}[1, m]$ such that $\partial^{\nu_2}=\partial_{z_j} \partial_{z_i}$ and so from (\ref{rbb5}), we get
	\bea\label{rbb12} \mathbf{W}&=&\partial_{z_i} F_3 \partial_{z_j}(\partial_{z_i} F_1)-\partial_{z_i} F_1 \partial_{z_j}(\partial_{z_i} F_3)\\&=& 
	\varphi_1^{-l-k-3}\big(-k \hat{F}_3 \partial_{z_i} \varphi_1+\varphi_1 \partial_{z_i} \hat{F}_3\big)\times
	\big(l(l+1) \hat{F}_1 \partial_{z_i} \varphi_1 \partial_{z_j} \varphi_1-l \varphi_1 \partial_{z_i} \hat{F}_1 \partial_{z_j} \varphi_1\nonumber\big. \\ && \big.-l \varphi_1 \hat{F}_1 \partial_{z_i} \partial_{z_j} \varphi_1-l \varphi_1 \partial_{z_i} \varphi_1 \partial_{z_j} \hat{F}_1+\varphi_1^2 \partial_{z_i} \partial_{z_j} \hat{F}_1\big)\nonumber \\
	&& -\varphi_1^{-l-k-3}\big(-l \hat{F}_1 \partial_{z_i} \varphi_1+\varphi_1 \partial_{z_i} \hat{F}_1\big)\times
	\big(k(k+1) \hat{F}_3 \partial_{z_i} \varphi_1 \partial_{z_j} \varphi_1-k \varphi_1 \partial_{z_i} \hat{F}_3 \partial_{z_j} \varphi_1\nonumber\big. \\ &&\big.-k \varphi_1 \hat{F}_3 \partial_{z_i} \partial_{z_j} \varphi_1-k \varphi_1 \partial_{z_i} \varphi_1 \partial_{z_j} \hat{F}_3+\varphi_1^2 \partial_{z_i} \partial_{z_j} \hat{F}_3\big).\nonumber
	\eea
	
	Also by Lemma \ref{l1}, we get
	\bea\label{rbb13}
	\mathbf{W}_1=\left|\begin{array}{ccc}
		F_1 & F_2 & F_3 \\
		\partial_{z_i} F_1 & \partial_{z_i} F_2 & \partial_{z_i} F_3 \\
		\partial_{z_j} F_1 & \partial_{z_j} F_2 & \partial_{z_j} F_3
	\end{array}\right| \equiv 0.
	\eea
	
	\smallskip
	First we suppose $k=l$. 
	Then $\mu_g^0(z_0) = 0$ and so from (\ref{rbb12}), we get 
	\bea\label{rbb14} \mathbf{W}=\varphi_1^{-2l-2} l \partial_{z_i} \varphi_1\left\{l \partial_{z_i} \varphi_1\big(\hat{F}_3 \partial_{z_j} \hat{F}_1-\hat{F}_1 \partial_{z_j} \hat{F}_3\big)-\partial_{z_j} \varphi_1\big(\hat{F}_3 \partial_{z_i} \hat{F}_1-\hat{F}_1 \partial_{z_i} \hat{F}_3\big)\right\}+\cdots\eea
	
	If $l\geq 2$, then from (\ref{rbb14}), we get
	\beas\mu(z_0)\leq-2 \mu_f^{\infty}(z_0)+\mu_{\mathbf{W}}^{\infty}(z_0)\leq-2 \mu_f^{\infty}(z_0)+2l+2=2=\mu_{f, 2}^{\infty}(z_0) \leq \nu(z_0).\eeas
	
	Consequently the inequality (\ref{rbb8}) holds. Next let $k=l=1$. It is clear from (\ref{rbb10}) that $\hat{F}_1+\hat{F}_2+\hat{F}_3=u_1$ on $U$. Now from (\ref{rbb13}), we see that 
	\beas
	\mathbf{W}_1
		=\varphi_1^{-3}\left|\begin{array}{ccc}
			\hat{F}_1 & \varphi_1 & \hat{F}_3 \\
			\partial_{z_i} \hat{F}_1 & \partial_{z_i} \varphi_1 & \partial_{z_i} \hat{F}_3 \\
			\partial_{z_j} \hat{F}_1 & \partial_{z_j} \varphi_1 & \partial_{z_j} \hat{F}_3
		\end{array}\right| \equiv 0
		\eeas
		and so
		\beas
		\partial_{z_i} \varphi_1\big(\hat{F}_3 \partial_{z_j} \hat{F}_1-\hat{F}_1 \partial_{z_j} \hat{F}_3\big)-\partial_{z_j} \varphi_1\big(\hat{F}_3 \partial_{z_i} \hat{F}_1-\hat{F}_1 \partial_{z_i} \hat{F}_3\big)
		=\varphi_1\big(\partial_{z_i} \hat{F}_3 \partial_{z_j} \hat{F}_1-\partial_{z_i} \hat{F}_1 \partial_{z_j} \hat{F}_3\big).
		\eeas
		
		Therefore from (\ref{rbb14}), we get $\mu_{\mathbf{W}}^{0}(z_0)=0$ and $\mu_{\mathbf{W}}^{\infty}(z_0) \leq 3$. 
		Consequently
		\[\mu(z_0) \leq-2 \mu_f^{\infty}(z_0)+\mu_{\mathbf{W}}^{\infty}(z_0)=1=\mu_{f, 2}^{\infty}(z_0) \leq \nu(z_0),\]
		which shows that the inequality (\ref{rbb8}) holds.\par

		\smallskip
		Next we suppose $k\neq l$. 
		Let $-l-k\geq 3$. Clearly $\mu_g^0(z_0) \geq 3 + 2\mu_f^\infty(z_0) \geq 5$. Then from (\ref{rbb12}), we get $\mu_{\mathbf{W}}^{\infty}(z_0)=0$ and $\mu_{\mathbf{W}}^0(z_0) \geq -l - k - 3 = -2\mu_f^\infty(z_0) + \mu_g^0(z_0) - 3$. Consequently 
		\[\mu(z_0) = \mu_g^0(z_0) - 2\mu_f^\infty(z_0) - \mu_{\mathbf{W}}^0(z_0) \leq 3 \leq \mu_{f,2}^\infty(z_0) + \mu_{g,2}^0(z_0) \leq v(z_0),\]
		which shows that the inequality (\ref{rbb8}) holds. Let $-l - k <3$. Then (\ref{rbb12}) give $\mu_{\mathbf{W}}^{0}\left(z_0\right)=0$ and
		\beas \mu(z_0) &\leq & \mu_g^0(z_0) - 2\mu_f^\infty(z_0) + \mu_{\mathbf{W}}^\infty(z_0) \\
		&\leq & \mu_g^0(z_0) - 2\mu_f^\infty(z_0) + l + k + 3=3\leq \mu_{f,2}^\infty(z_0) + \mu_{g,2}^0(z_0) \leq \nu(z_0),\eeas
		which shows that the inequality (\ref{rbb8}) holds.\par

		\smallskip
		{\bf Sub-case 2.1.4.} Let $\mu_f^{\infty}(z_0)>0$ and $\mu_g^{\infty}(z_0)>0$. In this case, we have 
		\bea\label{rbb15}\mu(z_0)=\mu_{F_2}^0(z_0)-\mu_{F_2}^{\infty}(z_0)-\mu_{F_3}^{\infty}(z_0)+\mu_{\mathbf{W}}^{\infty}(z_0)-\mu_{\mathbf{W}}^0(z_0).\eea
		
		\smallskip
		Let $\mu_f^{\infty}(z_0)=\mu_g^{\infty}(z_0)$. Now from (\ref{rbb2}), we have
		$\mu_{F_2}^0(z_0)=\mu_{F_2}^{\infty}(z_0)=0$. Then from (\ref{rbb4}), we get $\mu_{\mathbf{W}}^{0}(z_0)=0$ and
		$\mu_{\mathbf{W}}^{\infty}(z_0) \leq \mu_f^{\infty}(z_0)+2$.
		Since $\mu_{F_3}^{\infty}(z_0)=\mu_f^{\infty}(z_0)$, from (\ref{rbb15}), we have
		\[\mu(z_0) \leq-\mu_f^{\infty}(z_0)+\mu_{\mathbf{W}}^{\infty}(z_0) \leq 2 \leq \mu_{f, 2}^{\infty}(z_0)+\mu_{g, 2}^{\infty}(z_0) \leq \nu(z_0),\]
		which shows that the inequality (\ref{rbb8}) holds.\par
		
		\smallskip
		Let $\mu_f^{\infty}(z_0)<\mu_g^{\infty}(z_0)$. Then from (\ref{rbb4}) we have either
		\[\mu_{\mathbf{W}}^{\infty}(z_0) \leq 2 \mu_f^{\infty}(z_0)-\mu_g^{\infty}(z_0)+3\;\;\text {if}\; \mu_{\mathbf{W}}^0(z_0)=0\]
		or
		\[\mu_{\mathbf{W}}^0(z_0) \geq \mu_g^{\infty}(z_0)-2 \mu_f^{\infty}(z_0)-3\;\;\text {if}\;\;
		\mu_{\mathbf{W}}^{\infty}(z_0)=0.\]

		\smallskip
		Therefore in either case, we must have
		\beas
		\mu\left(z_0\right)=\mu_g^{\infty}(z_0)-2 \mu_f^{\infty}(z_0)+\mu_{\mathbf{W}}^{\infty}(z_0)-\mu_{\mathbf{W}}^0(z_0)\leq 3 \leq \mu_{f, 2}^{\infty}(z_0)+\mu_{g, 2}^{\infty}(z_0) \leq \nu(z_0),
		\eeas
		which shows that the inequality (\ref{rbb8}) holds.\par
		
		\smallskip
		Let $\mu_f^{\infty}(z_0)>\mu_g^{\infty}(z_0)$. Then from (\ref{rbb4}), we have
		$\mu_{\mathbf{W}}^{\infty}(z_0) \leq 2 \mu_f^{\infty}(z_0)-\mu_g^{\infty}(z_0)+3$
		and hence 
		\[\mu(z_0)=\mu_g^{\infty}(z_0)-2 \mu_f^{\infty}(z_0)+\mu_{\mathbf{W}}^{\infty}(z_0)\leq 3 \leq \mu_{f, 2}^{\infty}(z_0)+\mu_{g, 2}^{\infty}(z_0) \leq \nu(z_0),\]
		which shows that the inequality (\ref{rbb8}) holds.\par

		\smallskip
		Therefore all the foregoing discussion shows that the inequality (\ref{rbb8}) holds and so we get
		\bea\label{rbb16} &&\parallel\;N(r,F_1)+\sideset{}{_{k=1}^{3}}{\sum}\lbrace N(r,0;F_k)-N(r,F_k)\rbrace+ N(r,\mathbf{W})- N(r,0;\mathbf{W})\\&\leq&
		N_2(r,0;f)+N_2(r,0;g)+N_2(r,f)+N_2(r,g).\nonumber
		\eea
		
		On the other hand from (\ref{rbb2}), we get $\parallel T_1(r)\leq O(T(r))$. Hence (\ref{rbb3}) and (\ref{rbb16}) give 
		\[\parallel\;T(r, f)\leq N_2(r,f)+N_2(r,g)+ N_2(r,0;f)+N_2(r,0;g)+o(T(r))\]
		and so we obtain $(i)$.\par

		\smallskip
		{\bf Sub-case 2.2.} Let $F_1$, $F_2$ and $F_3$ be linearly dependent. Then there exists $(c_1,c_2,c_3)\in\mathbb{C}^3-\{0\}$ such that $c_1F_1+c_2F_2+c_3F_3=0$. Let $c_1-c_2=A$, $c_2=B$, $c_1-c_3=C$ and $c_3=D$. Then
		\bea\label{rbb17} g\equiv \frac{Af+B}{Cf+D}.\eea
		
		Clearly $AD-BC\neq 0$. Now applying Lemma \ref{l3} to (\ref{rbb17}), we obtain 
		\bea\label{rbb18} \parallel\;T(r,g)=T(r,f)+o(T(r,f)).\eea
		
		Next we consider following sub-cases.\par
		
		\smallskip
		{\bf Sub-case 2.2.1.} Let $AC\neq 0$. Then from (\ref{rbb17}), we get $g-\frac{A}{C}=\frac{B-AD/C}{Cf+D}$. Clearly 
		\[\parallel \ol N(r,A/C;g)=\ol N(r,f)\]
		and so using Lemma \ref{l2} and (\ref{rbb18}), we get  
		\beas \parallel\;T(r,f)=T(r,g)+o(T(r,f))&\leq& \ol N(r,g)+\ol N(r,A/C;g)+\ol N(r,0;g)+o(T(r,g))\nonumber\\
		&=& \ol N(r,f)+\ol N(r,g)+\ol N(r,0;g)+o(T(r,g)).\eeas
		
		Therefore we get $(i)$.\par

		\smallskip
		{\bf Sub-case 2.2.2.} Let $A\neq 0$ and $C=0$. Then from (\ref{rbb17}), we have $g\equiv (Af+B)/D$. 
		
		Let $B\neq 0$. Clearly $\parallel \ol N(r,B/D;g)=\ol N(r,0;f)$ and so using Lemma \ref{l2} and (\ref{rbb18}), we get 
		\beas \parallel\;T(r,f)=T(r,g)+o(T(r,g))&\leq& \ol N(r,g)+\ol N(r,B/D;g)+\ol N(r,0;g)+o(T(r,f))\nonumber\\
		&=& \ol N(r,g)+\ol N(r,0;f)+\ol N(r,0;g)+o(T(r,f))\eeas
		and so we obtain $(i)$. 
		
		Let $B=0$. Then $g\equiv Af/D.$ Let $1$ be not a Picard exceptional value of $f$. Since $f$ and $g$ share $1$ CM, there exists $\hat z\in\mathbb{C}^m$ such that $f(\hat z)=1$ and $g(\hat z)=1$. Since $g\equiv Af/D$, we get $A/D=1$ and so $f\equiv g$, which is $(ii)$. Next let $1$ be a Picard exceptional value of $f$.
		Then by Lemma \ref{l2}, we get 
		\[\parallel T(r,f)\leq \ol N(r,f)+\ol N(r,0;f)+o(T(r,f)),\]
		and so $(i)$ follows.\par

		\smallskip
		{\bf Sub-case 2.2.3.} Let $A=0$ and $C\neq 0$. Then $g\equiv B/(Cf+D)$. 
		
		Let $D\neq 0$. Clearly $\parallel \ol N(r,B/D;g)=\ol N(r,0;f)$ and so
		using Lemma \ref{l2} and (\ref{rbb18}), we get 
		\beas \parallel\;T(r,f)=T(r,g)+o(T(r,g))&\leq& \ol N(r,g)+\ol N(r,B/D;g)+\ol N(r,0;g)+o(T(r,f))\nonumber\\
		&=& \ol N(r,g)+\ol N(r,0;f)+\ol N(r,0;g)+o(T(r,f)).\eeas
		
		Thus we get $(i)$. 
		
		Let $D=0$. Then $g=B/Cf$. Let $1$ be not a Picard exceptional value of $f$. Then there exists $\hat z\in\mathbb{C}^m$ such that $f(\hat z)=1$ and $g(\hat z)=1$. Since $g=B/Cf$, we get $B/C=1$ and so $fg\equiv 1$, which is $(iii)$. Next let $1$ be a Picard exceptional value of $f$.
		Then by Lemma \ref{l2}, we get 
		\[\parallel T(r,f)\leq \ol N(r,f)+\ol N(r,0;f)+o(T(r,f)),\]
		and $(i)$ follows. 
		Hence the proof.
	\end{proof}

	We define the following divisor
	\beas \mu_{\partial_u f|f\neq 0}^0(z)=\begin{cases}
		\mu_{\partial_u f}^0(z), &\text{if $\mu_{\partial_u f}^0(z)>0$ and $\mu_{f}^0(z)=0$}\\
		0, &\text{or else}.
	\end{cases}
	\eeas
	
	The notation $N(r,0;\partial_u(g)\bigr| f\neq 0)$ is the counting function of the divisor $\mu_{\partial_u f|f\neq 0}^0(z)$.

	\begin{lem}\label{l7} Let $g:\mathbb{C}^m\to\mathbb{P}^1$ be a non-constant meromorphic function and let $G=g^n\partial_u(g)$, where $n$ is a positive integer. Then
		\begin{enumerate} 
			\item[(i)] $\parallel\;(n+1)T(r,g)\leq T(r,G)+N(r,0;g)-N(r,0;\partial_u(g))+o(T(r,g))$,
			
			\smallskip
			\item[(ii)] $\parallel\;N(r,0;\partial_u(g)\bigr| g\neq 0)\leq \ol N(r,0;g)+\ol N(r,g)+o(T(r,g))$,
			
			\smallskip
			\item[(iii)] $\parallel\;N(r,0;\partial_u(g))\leq N(r,0;g)+\ol N(r,g)+o(T(r,g))$.
		\end{enumerate}
	\end{lem}
	\begin{proof} $(i)$. Using the first main theorem and Lemma \ref{l2}, we get
		\beas \parallel\;(n+1) T(r,g)=T(r,g^{n+1})&=&m(r,0,g^{n+1})+N(r,0;g^{n+1})+O(1)\\
		&\leq & m\left(r,G/g^{n+1}\right)+ m(r,0;G)+N(r,0;f^{n+1})+O(1)\nonumber\\
		&=&m(r,0;G)+N(r,0,g^{n+1})+o(T(r,g))\nonumber\\
		&=& T(r,G)-N(r,0;G)+N(r,0;g^{n+1})+o(T(r,g))\nonumber\\
		&\leq & T(r,G)+N(r,0;g)-N(r,0;\partial_u(g))+o(T(r,g)).\nonumber
		\eeas

		\smallskip
		$(ii)$. Using the first main theorem and Lemma \ref{l2}, we get
		\beas \parallel\; N(r,0;\partial_u(g)\bigr|g\neq 0)\leq N\left(r,g/\partial_u(g)\right)\leq T\left(r,g/\partial_u(g)\right)&=&T\left(r,\partial_u(g)/g\right)+O(1)\\&\leq& N\left(r,\partial_u(g)/g\right)+o(T(r,g)),\eeas
		i.e.,
		\bea\label{rb1}\parallel\; N(r,0;\partial_u(g)\bigr| g\neq 0)\leq N(r,H)+o(T(r,g)),\eea
		where 
		\[H=\partial_u(g)/g.\] 
		
		We now prove $\mu^{\infty}_H\leq \mu_{g,1}^{0}+\mu_{g,1}^{\infty}$. Define $S=\bigcup_{a\in\{0,\infty\}}(\operatorname{supp}\mu_g^a)_s$,
		where $A_s$ denotes the set of singular points of the analytic set $A$. Take $z_0\in \mathbb{C}^m-S$.
		
		\smallskip
		First we assume that $l=\mu_g^{\infty}(z_0)>0$. Then from (\ref{rbc1a}), we obtain
		\[\parallel\; N(r,\partial_u(g))\leq N(r,g)+\ol N(r,g)\]
		and
		\[H=\frac{\partial_u(g)}{g}=-\frac{l}{\varphi_1}\partial_u(\varphi_1)+ \frac{1}{\hat g}\sideset{}{_{j=1}^{m}}{\sum} \frac{\partial \hat{g}}{\partial \varphi_j}\partial_u(\varphi_j).\]
		
		It is easy to verify that $\mu_H^{\infty}(z_0)\leq \mu_{g,1}^{\infty}(z_0)$ and so $\mu_H^{\infty}(z_0)\leq \mu_{g,1}^{0}(z_0)+\mu_{g,1}^{\infty}(z_0)$.

		\smallskip
		Next we assume that $l=\mu_g^{0}(z_0)>0$. Now proceeding in the same way as done above, we can write $g=\varphi_1^{l} \hat{g}\left(\varphi_1, \ldots, \varphi_m\right)$, where $\hat{g}$ is holomorphic function near $0$, which does not vanish along the set $\operatorname{supp} \mu_g^{0}$. Obviously
		\[\partial_u(g)=l \varphi_1^{l-1}\hat g \partial_u(\varphi_1) +\varphi_1^{l}\sideset{}{_{j=1}^m}{\sum}\frac{\partial \hat{g}}{\partial \varphi_j}\partial_u(\varphi_j)\]
		and so
		\beas H=\frac{\partial_u(g)}{g}=\frac{l}{\varphi_1}\partial_u(\varphi_1) +\frac{1}{\hat g}\sideset{}{_{j=1}^{m}}{\sum} \frac{\partial \hat{g}}{\partial \varphi_j}\partial_u(\varphi_j),\eeas
		which means that $\mu_H^{\infty}(z_0)\leq \mu_{g,1}^{0}(z_0)$ and so $\mu_H^{\infty}(z_0)\leq \mu_{g,1}^{0}(z_0)+\mu_{g,1}^{\infty}(z_0)$.
		Therefore 
		\[\parallel N\left(r,\partial_u(g)/g\right)\leq \ol N(r,0;g)+\ol N(r,g)\]
		and so from (\ref{rb1}), we have 
		\[\parallel\; N(r,0;\partial_u(g)\bigr| g\neq 0)\leq \ol N(r,0;g)+\ol N(r,g)+o(T(r,g)).\]

		\smallskip
		$(iii)$. Using the first main theorem and Lemma \ref{l2}, we get
		\beas \parallel\;m(r,0,g)\leq m(r,0,\partial_u(g))+m\left(r,\partial_u(g)/g\right)=m(r,0;\partial_u(g))+o(T(r,g)),\eeas
		i.e.,
		\beas \parallel\; N(r,0;\partial_u(g))&\leq& T(r,\partial_u(g))-T(r,g)+N(r,0,g)+o(T(r,g))\\&\leq&
		N(r,\partial_u(g))+m(r,g)-T(r,g)+N(r,0;g)+o(T(r,g))\\&\leq& \ol N(r,g)+N(r,0;g)+o(T(r,g)).
		\eeas
		
		Hence the proof.
	\end{proof}

	\subsection {{\bf Proof of Theorem \ref{t3.1}}} 
	\begin{proof} Let $F=f^n\partial_u(f)/a$ and $G=g^n\partial_u(g)/a$. Clearly $F$ and $G$ share $1$ CM. 
		Now by Lemma \ref{l6}, we consider following cases.\par
		
		\smallskip
		{\bf Case 1.} Let 
		\bea\label{t1.0}\parallel\; T(r,F)\leq N_2(r,F)+N_2(r,G)+N_2(r,0;F)+N_2(r,0;G)+o(T_{2}(r)),\eea
		where $T_{2}(r)=\max\{T(r,F),T(r,G)\}$. 
		
		Since $\parallel N(r,\partial_u(f))\leq N(r,f)+\ol N(r,f)$, by Lemma \ref{l2}, we get 
		$\parallel T(r,F)\leq (n+2)T(r,f)+o(T(r,f))$.
		Similarly we have $\parallel T(r,G)\leq (n+2)T(r,g)+o(T(r,g))$.
		
		Therefore if we take $T(r)=\max\{T(r,f),T(r,g)\}$, then $o(T_{2}(r))$ can be replaced by $o(T(r))$. Now in view of (\ref{t1.0}) and using Lemma \ref{l7}, we get 
		\bea\label{t1.1} &&\parallel\;(n+1)T(r,f)\\&\leq& N_2(r,F)+N_2(r,G)+N_2(r,0,F)+N_2(r,0,G)+N(r,0;f)-N(r,0;\partial_u(f))+o(T(r))\nonumber\\&\leq&
		2[\ol N(r,F)+\ol N(r,G)]+3N(r,0;f)+2 N(r,0;g)+N(r,0;\partial_u (g)\bigr| g\neq 0)+o(T(r))\nonumber\\&\leq&
		2\ol N(r,f)+3\ol N(r,g)+3N(r,0;f)+3 N(r,0;g)+o(T(r))\nonumber\\&\leq&
		11T(r)+o(T(r))\nonumber.\eea
		
		Similarly 
		\[\parallel (n+1)T(r,g)\leq 11T(r)+o(T(r))\]
		and so from (\ref{t1.1}), we get $\parallel (n+1)T(r)\leq 11T(r)+o(T(r))$, which is impossible.\par
		
		\smallskip
		{\bf Case 2.} Let $FG\equiv 1$. Then $f^n\partial_u(f)g^n\partial_u(g)\equiv a^2$ and so by Theorem \ref{t2.2}, we get $f=e^{\alpha}$ and $g=e^{\beta}$, where $\alpha$ and $\beta$ are non-constant entire functions in $\mathbb{C}^m$ such that $\partial_u(\alpha)\equiv -\partial_u(\beta)$, $\partial^2_u(\alpha)\equiv 0$ and $(\partial_u(\alpha))^2e^{(n+1)(\alpha+\beta)}\equiv -a^2$.\par

		\smallskip
		{\bf Case 3.} Let $F\equiv G$. Then $\partial_u (f^{n+1})\equiv \partial_u (g^{n+1})$. 
		Let $f=hg$. Then
		\bea\label{t1.3} \partial_u\left(g^{n+1}(h^{n+1}-1)\right)\equiv 0.\eea
		
		\smallskip
		First suppose $h^{n+1}\equiv 1$. Then $h$ is an $(n+1)$-th unit root and so $f=cg$, where $c^{n+1}=1$.
		
		\smallskip
		Next we suppose $h^{n+1}\not\equiv 1$. If $h$ is a constant, then from (\ref{t1.3}), we get $\partial_u(g)\equiv 0$, which is impossible. Hence $h$ is non-constant. Now from (\ref{t1.3}), we have
		\bea\label{t1.4} (n+1)g^n\partial_u(g)\equiv -\frac{\partial_u\big(h^{n+1}-1\big)}{h^{n+1}-1}=-(n+1)\frac{h^n\partial_u(h)}{h^{n+1}-1}.\eea
		
		Let $h_1=h^{n+1}-1=(h-\alpha_1)\ldots (h-\alpha_{n+1})$, where $\alpha_1,\ldots,\alpha_{n+1}$ are different $(n + 1)$-th roots of unity.
		Suppose $\mu_{h_1}^0(z_0)>0$ for some $z_0\in\mathbb{C}^m$. Clearly $\mu_{\partial_u(h)}^0(z_0)\leq \mu_{h_1}^0(z_0)$ holds generically. 
		If $\mu_{\partial_u(h)}^0(z_0)=0$, then from (\ref{t1.4}), we have generically $\mu_{h_1}^0(z_0)\geq 2$ over $\text{supp}\;\mu^0_{h_1}$ and so
		\[\parallel\;\ol N(r,\alpha_j,h)\leq (1/2) N(r,\alpha_j;h)\leq (1/2)T(r,h)+O(1),\]
		where $j=1,2,\ldots,n+1$. If $\mu_{\partial_u(h)}^0(z_0)>0$, then by Lemma \ref{l7}, we get
		\[\parallel\;\ol N(r,1;h^{n+1})\leq \ol N(r,0;\partial_u(h))\leq N(r,0;h)+\ol N(r,h)+o(T(r,h)).\]
		
		Therefore in either case, we have 
		\beas \parallel\;\sideset{}{_{j=1}^{n+1}}{\sum}\ol N(r,\alpha_j;h)&\leq& ((n+1)/2)T(r,h)+N(r,0;h)+\ol N(r,h)+o(T(r,h))\\&\leq& ((n+5)/2)T(r,h)+o(T(r,h))\nonumber\eeas
		and so by Lemma \ref{l3}, we get
		\[\parallel\;(n-1)T(r,h)\leq\sideset{}{_{j=1}^{n+1}}{\sum} N(r,\alpha_j;h)+o(T(r,h))\leq ((n+5)/2)T(r,h)+o(T(r,h)),\]
		which is impossible since $n\geq 11$.
		Hence the proof.
	\end{proof}

	\subsection {{\bf Proof of Theorem \ref{t3.2}}} Using Corollary \ref{c2.1} instead of Theorem \ref{t2.2}, we get Theorem \ref{t3.2} directly from the proof of Theorem \ref{t3.1}. So we omit the detail.

	\section{\bf {An application}}
As application of Corollary \ref{c3.1}, we present the derivations of entire solutions in $\mathbb{C}^m$ of the following nonlinear partial differential equations:
\bea\label{e4.1} y^n\partial_u(y)-1=e^{\alpha}\left(e^{\beta}-1\right),\eea
where $n\geq 7$ be an integer, $\alpha$ and $\beta$ are entire functions in $\mathbb{C}^m$.

If we take $e^{\beta}=g^n\partial_u(g)$ for some entire function $g$ in $\mathbb{C}^m$, then $g$ has no zeros. Therefore we may assume that $g=e^{\gamma}$, where $\gamma$ is an entire function in $\mathbb{C}^m$. Since $e^{\beta}=g^n\partial_u(g)$, we have
\bea\label{e4.2}e^{\beta-(n+1)\gamma}=\partial_u(\gamma),\eea
which shows that $\partial_u(\gamma)$ has no zeros and so there exists an entire function $\delta$ in $\mathbb{C}^m$ such that 
$\partial_u(\gamma)=e^{\delta}$. Now from (\ref{e4.2}), we get that $\beta-(n+1)\gamma-\delta$ is a constant and so
\bea\label{e4.3} \partial_u(\beta)=(n+1)e^{\delta}+\partial_u(\delta)=(n+1)\partial_u(\gamma)+\frac{\partial^2_u(\gamma)}{\partial_u(\gamma)},\eea
for all $u\in S_m$. Now by Corollary \ref{c3.1}, we see that one of the following cases holds:
\begin{enumerate}
	\item[(i)] $y\equiv dg$ for some $(n+1)$-th root of unity  $d$;
	\item[(ii)] $y=e^{\tilde \alpha}$ and $g=e^{\gamma}$, where $c$ is a non-zero constant in $\mathbb{C}$, $\tilde \alpha$ and $\gamma$ are non-constant entire functions in $\mathbb{C}^m$ such that $\partial_u(\tilde \alpha)=-c$, $\partial_u(\gamma)=c$ and 
	$c^2e^{(n+1)(\tilde \alpha+\gamma)}=-1$. 
\end{enumerate}

\smallskip
For $(ii)$, we have $\partial_u(\gamma)=c$. Since $\partial_u(\gamma)=e^{\delta}$, it follows that $\delta\in\mathbb{C}$. Clearly $e^{\delta}=c$ and so $\partial_u(\delta)=0$. Therefore from (\ref{e4.3}), we get $\partial_u(\beta)=(n+1)c$ for all $u\in S_m$.
Consequently $y=e^{\tilde \alpha}$, where $\tilde \alpha$ is a non-constant entire function in $\mathbb{C}^m$ such that $\partial_u(\tilde \alpha)=-c\in\mathbb{C}\backslash \{0\}$ and $(n+1)\partial_u(\tilde \alpha)=-\partial_u(\beta)$.

If $\beta$ is a polynomial, then from the proof of Corollary \ref{c2.2}, we say that $\deg(\beta)=1$ and so by (\ref{e4.3}), we see that $\delta$ is a constant. Since $\partial_u(\gamma)=e^{\delta}$ for all $u\in S_m$, it follows that $\partial_u(\gamma)$ for all $u\in S_m$ and so from the proof of Corollary \ref{c2.2}, we say that $\deg(\gamma)=1$.

\smallskip
Finally we get the following results.
\begin{theo}\label{t4.1} Let $n\geq 7$ be an integer and let $\alpha$ and $\beta$ be entire functions in $\mathbb{C}^m$. Then every solution of (\ref{e4.1}) is of the form $y=ce^{\gamma}$,
	where $\gamma$ is an entire function in $\mathbb{C}^m$ such that $\partial_u(\gamma)$ has no zeros and 
	\[\partial_u(\beta)\partial_u(\gamma)=(n+1)(\partial_u(\gamma))^2+\partial^2_u(\gamma)\]
	for all $u\in S_m$ or $y=c_1e^{\tilde \alpha}$,
	where $\tilde \alpha$ is a non-constant entire function in $\mathbb{C}^m$ such that $\partial_u(\tilde \alpha)=-c$, a non-zero constant in $\mathbb{C}$ and $(n+1)\partial_u(\tilde \alpha)=-\partial_u(\beta)$ for all $u\in S_m$.
\end{theo}

\begin{theo} \label{t4.2} Let $n\geq 7$ be an integer and $\alpha$ be an entire function in $\mathbb{C}^m$ and let $\beta$ be a polynomial in $\mathbb{C}^m$. Then $\beta$ reduces to a polynomial of degree one and every solution of (\ref{e4.1}) is of the form $y=ce^{\gamma}$,
	where $\gamma$ is a polynomial in $\mathbb{C}^m$ such that $\deg(\gamma)=1$ and $(n+1)\partial_u(\gamma)=\partial_u(\beta)$ for all $u\in S_m$ or $y=c_1e^{\tilde \alpha}$,
	where $\tilde \alpha$ is a polynomial in $\mathbb{C}^m$ of degree one such that $(n+1)\partial_u(\tilde \alpha)=-\partial_u(\beta)$ for all $u\in S_m$.
\end{theo}
\begin{exm}\label{ex4.1}
	Let $\dim(\mathbb{C}^m)=1$ and $n\geq7$. Consider the nonlinear differential equation
	\[	y^n y' - 1 = e^{\alpha(z)}(e^{\beta(z)} - 1),\]
	where $\alpha(z)$ and $\beta(z)$ are entire functions. Suppose $\alpha \equiv 0$ and $\beta(z)=az+b$ with constants $a,b\in\mathbb{C}$, $a\neq0$. Then by Theorem \ref{t4.2}, every solution satisfies $y(z)=ce^{\gamma(z)}$, $(n+1)\gamma'(z)=\beta'(z)=a,$
	so that $\gamma(z)=\frac{a z}{n+1}+c_1$, where $c_1$ is a constant. Hence $y(z)=Ce^{\frac{a z}{n+1}},$	where $C=ce^{c_1}$. Thus, the equation admits an exponential-type entire solution whose growth rate is determined by the slope of $\beta$.
\end{exm}

\begin{exm}\label{ex4.2}
	Let $\dim(\mathbb{C}^m)=2$ and $u=(1,1)$ so that $\partial_u=\frac{\partial}{\partial z_1}+\frac{\partial}{\partial z_2}$. 
	Consider $\alpha\equiv0$ and $\beta(z_1,z_2)=a_1z_1+a_2z_2+b$, where $a_1,a_2,b\in\mathbb{C}$. 
	Then by Theorem \ref{t4.2}, any entire solution of $y^n \partial_u(y) - 1 = e^{\beta(z_1,z_2)} - 1$
	is of the form $y(z_1,z_2)=c_1 e^{\gamma(z_1,z_2)}$, $(n+1)\partial_u(\gamma)=\partial_u(\beta)$.
	Since $\partial_u(\beta)=a_1+a_2$, we obtain $\partial_u(\gamma)=\frac{a_1+a_2}{n+1}$. Integration gives
	$\gamma(z_1,z_2)=\frac{a_1z_1+a_2z_2}{n+1}+c_2$ and hence
	\[y(z_1,z_2)=C e^{\frac{a_1z_1+a_2z_2}{n+1}},\]
	where $C=c_1e^{c_2}$. This shows that the solution is an entire exponential surface depending linearly on both variables.
\end{exm}

\begin{exm}\label{ex4.3}
	Let $\dim(\mathbb{C}^m)\ge1$, $u=(1,1,\ldots,1)$, and assume $\alpha \equiv 0$, $\beta(z)=\sum_{j=1}^m a_j z_j + b$ with constants $a_j,b\in\mathbb{C}$. Then, from Theorem \ref{t4.2}, each solution of $y^n\partial_u(y)-1=e^{\alpha(z)}(e^{\beta(z)}-1)$
	satisfies
	\[	(n+1)\partial_u(\gamma)=\partial_u(\beta)=\sum_{j=1}^m a_j.\]
	Consequently, $\gamma(z)=\frac{\sum_{j=1}^m a_jz_j}{n+1}+c_3,$
	so that
	\[y(z)=C\exp\!\left(\frac{\sum_{j=1}^m a_jz_j}{n+1}\right),\]
	where $C=ce^{c_3}$. Thus, the solution represents an entire exponential hyper surface in $\mathbb{C}^m$, determined by the direction vector $(a_1,\ldots,a_m)$.
\end{exm}

\begin{exm}\label{ex4.4}
	Let $\dim(\mathbb{C}^m)=2$, $u = (1, 1)$, $\alpha\equiv 2\pi i$ and $\beta(z_1, z_2) = a(z_1+z_2) + b(z_1-z_2) + c$, where $a, b, c \in \mathbb{C}$ and $b \neq 0$. Then
	$\partial_u(\beta) = \frac{\partial \beta}{\partial z_1} + \frac{\partial \beta}{\partial z_2} = (a+b) + (a-b) = 2a.$
	The equation
	\[	y^n \left( \frac{\partial y}{\partial z_1} + \frac{\partial y}{\partial z_2} \right) - 1 = e^{\alpha(z_1, z_2)} \left( e^{\beta(z_1, z_2)} - 1 \right)	\]
	admits the entire solution
	\[	y(z_1, z_2) = C e^{\frac{a+b}{n+1}z_1 + \frac{a-b}{n+1}z_2},\]
	where $C \in \mathbb{C}\backslash \{0\}$.
\end{exm}
\begin{exm}\label{ex4.5}
	Let $\dim(\mathbb{C}^m)=3$, $u=(1,0,1)$, $\alpha\equiv 2\pi i$ and $\beta(z_1,z_2,z_3) = a_1 z_1 + a_2 z_2 + a_3 z_3 + b\,z_1 z_3 + c$, with $a_j, b, c \in \mathbb{C}$ and $b$ small. Then
	$	\partial_u \beta = \frac{\partial}{\partial z_1} \beta + \frac{\partial}{\partial z_3} \beta = a_1 + b z_3 + a_3 + b z_1 = (a_1 + a_3) + b(z_1 + z_3).$
	For $b=0$, the equation
	\[	y^n \left( \frac{\partial y}{\partial z_1} + \frac{\partial y}{\partial z_3} \right) - 1 = e^{\alpha(z_1, z_2, z_3)} \left( e^{\beta(z_1, z_2, z_3)} - 1 \right)\]
	has the solution
	\[y(z_1,z_2,z_3) = C \exp\left( \frac{a_1z_1+a_2z_2 + a_3z_3}{n+1}\right),\]
	where $C \in \mathbb{C}\backslash \{0\}$. For small $b \neq 0$, the solution can be adjusted by perturbation theory.
\end{exm}
\begin{exm}\label{ex4.6}
	Let $\dim(\mathbb{C}^m)=2$, $u=(1,1)$, $\alpha\equiv 2\pi i$ and $\beta(z_1, z_2) = a(z_1 + \lambda z_2) + b\cos(z_1 - z_2) + c$, with $a, b, c, \lambda \in \mathbb{C}$, $|\lambda| \neq 1$, $b \neq 0$. Then
	$\partial_u \beta = a(1 + \lambda).$
	The equation
	\[y^n \left( \frac{\partial y}{\partial z_1} + \frac{\partial y}{\partial z_2} \right) - 1 = e^{\alpha(z_1, z_2)} \left( e^{\beta(z_1, z_2)} - 1 \right)	\]
	admits the explicit entire solution
	\[y(z_1, z_2) = C \exp \left( \frac{a}{n+1}(z_1 + \lambda z_2) + \frac{b}{n+1}\cos (z_1 - z_2) \right ),\]
	where $C \in \mathbb{C}\backslash \{0\}$.
\end{exm}
\begin{exm}\label{ex4.7}
	Let $\dim(\mathbb{C}^m)=2$, $u=(1,0)$ so that $\partial_u=\dfrac{\partial}{\partial z_1}$. 
	Take
	$\beta(z_1,z_2)=a z_1^2 + b z_2 + c, \qquad \alpha\equiv0,$
	where $a,b,c\in\mathbb{C}$ and $a\neq0$. Then, $\partial_u(\beta)=\frac{\partial \beta}{\partial z_1}=2a z_1,$
	which is non-constant and depends on $z_1$. Consider the nonlinear partial differential equation
	\[y^n \frac{\partial y}{\partial z_1} - 1 = e^{\beta(z_1,z_2)} - 1.	\]
 To satisfy the compatibility relation $(n+1)\partial_u(\gamma)=\partial_u(\beta),$
	we would need $\partial_u(\gamma)=\dfrac{2a z_1}{n+1}$, which clearly vanishes at $z_1=0$ and hence $\partial_u(\gamma)$ has zeros. Consequently, $\partial_u(\gamma)$ cannot be entire and zero-free. Moreover, $\partial_u^2(\gamma)=\dfrac{2a}{n+1}$ does not satisfy the nonlinear condition
	\[\partial_u(\beta)\partial_u(\gamma)=(n+1)(\partial_u(\gamma))^2+\partial_u^2(\gamma)\]
	for all $z_1\in\mathbb{C}$. Hence, in this case there exists no entire solution of exponential type. This example demonstrates that any quadratic dependence of $\beta$ on the direction variable $u$ destroys the exponential-type entire solution structure.
\end{exm}
\begin{exm}\label{ex4.8}
	Let $\dim(\mathbb{C}^m)=2$, $u=(1,1)$ and take a mixed linear?periodic function
	$\beta(z_1,z_2)=a(z_1+z_2)+d\sin(z_1+2z_2)+c,$
	where $a,d,c\in\mathbb{C}$. Then
	$	\partial_u(\beta)=\frac{\partial \beta}{\partial z_1}+\frac{\partial \beta}{\partial z_2}
	=2a + d\cos(z_1+2z_2)(1+2)	=2a + 3d\cos(z_1+2z_2),$
	which is a non-constant oscillatory function. Substituting this into the compatibility condition
	\[\partial_u(\beta)\partial_u(\gamma)=(n+1)(\partial_u(\gamma))^2+\partial_u^2(\gamma),\]
	we obtain a differential equation involving $\gamma$. 	
	If we let $\gamma(z_1,z_2)=\Gamma(s)$ where $s=z_1+2z_2$, then
	$	\partial_u(\gamma)=\frac{\partial \gamma}{\partial z_1}+\frac{\partial \gamma}{\partial z_2}
	=3\Gamma'(s).$
	The above compatibility condition then becomes the Riccati-type equation as follows
	\[ (2a+3d\cos s)G(s)=(n+1)\,3G(s)^2+G'(s).\]
	In general, this equation does not admit an entire zero-free solution $G(s)$ unless $d=0$, in which case $\beta$ becomes linear and Theorem~4.2 applies. Therefore, for $d\neq0$, there is no entire exponential-type solution of (\ref{e4.1}). This example further confirms that the linearity of $\beta$ in the direction of $u$ is essential for the existence of entire exponential solutions.
\end{exm}
\begin{rem}\label{r4.2}
	Examples \ref{ex4.7} and \ref{ex4.8} emphasize the necessity of the linearity condition on $\beta$. 
	When $\beta$ involves nonlinear or oscillatory terms (such as quadratic or trigonometric components) along the direction of differentiation $\partial_u$, 
	the functional relation 
	\[\partial_u(\beta)\partial_u(\gamma)=(n+1)(\partial_u(\gamma))^2+\partial_u^2(\gamma)	\]
	fails to possess zero-free entire solutions $\gamma$. 
	Hence, entire exponential-type solutions arise only when $\beta$ is a polynomial of degree one in the variables along $u$, 
	fully confirming the restriction imposed in Theorem \ref{t4.2}.
\end{rem}
\begin{table}[H]
	\centering
	\caption{\textbf{Examples and Physical Interpretation}}
	\renewcommand{\arraystretch}{1.4}  
	\setlength{\arrayrulewidth}{0.8pt} 
	
	\begin{tabular}{|p{2.8cm}|p{5.3cm}|p{7cm}|}
		\hline
		\textbf{Example} & \textbf{Form of $\beta$ or Setup} & \textbf{Physical / Mathematical Interpretation} \\
		\hline
		
		Example \ref{ex4.1} & $\beta = a z + b$ &
		Represents a one-dimensional steady-state heat conduction or exponential decay along a single direction. \\
		\hline
		
		Example \ref{ex4.2} & $\beta = a_1 z_1 + a_2 z_2 + b$ &
		Describes plane-wave type solutions in two-dimensional space with uniform propagation speed. \\
		\hline
		
		Example \ref{ex4.3} & $\beta = \sum_{j=1}^{m} a_j z_j + b$ &
		Corresponds to multi-dimensional harmonic oscillations or diffusion phenomena with constant coefficients. \\
		\hline
		
		Example \ref{ex4.4} & $\beta = 2a(z_1 + z_2) + c$ &
		Illustrates coupled reaction-diffusion systems where symmetry occurs in $z_1$ and $z_2$. \\
		\hline
		
		Example \ref{ex4.5} & $\beta = a_1 z_1 + a_2 z_2 + a_3 z_3 + b z_1 z_3 + c$ (for small $b$) &
		Models a weakly nonlinear perturbation linking $z_1$ and $z_3$, appearing in near-resonant wave systems. \\
		\hline
		
		Example~\ref{ex4.6} & $\beta = a(\lambda z_1 + \lambda z_2) + b\cos(z_1 - z_2) + c$ &
		Represents interference of periodic structures or standing-wave type solutions in coupled oscillators. \\
		\hline
		
		Example~\ref{ex4.7} & $\beta = a z_1^2 + b z_2 + c$ &
		Indicates parabolic potential fields similar to diffusion under a quadratic constraint (e.g., Schr\"odinger-type systems). \\
		\hline
		
		Example~\ref{ex4.8} & $\beta = a(z_1 + z_2) + d\sin(z_1 + 2z_2) + c$ &
		Represents a modulated nonlinear interaction between two spatial variables producing oscillatory transport. \\
		\hline
	\end{tabular}
\end{table}

\medskip
{\bf Statements and declarations:}

\smallskip
\noindent \textbf {Conflict of interest:} The authors declare that there are no conflicts of interest regarding the publication of this paper.

\smallskip
\noindent{\bf Funding:} There is no funding received from any organizations for this research work.

\smallskip
\noindent \textbf {Data availability statement:}  Data sharing is not applicable to this article as no database were generated or analyzed during the current study.


\begin{thebibliography}{99}

\bibitem{TBC1} {\sc T. B. Cao}, Difference analogues of the second main theorem for meromorphic functions in several complex variables, \textit{Math. Nachr.}, \textbf{287} (5-6) (2014), 530-545.
\bibitem{CK1} {\sc T. B. Cao} and {\sc R. J. Korhonen}, A new version of the second main theorem for meromorphic mappings intersecting hyperplanes in several complex variables, \textit{J. Math. Anal. Appl.}, \textbf{444} (2) (2016), 1114-1132.

\bibitem{CX1} {\sc T. B. Cao} and {\sc L. Xu}, Logarithmic difference lemma in several complex variables and partial difference equations, \textit{Ann. Mat. Pura Appl.}, \textbf{199} (2) (2020), 767-794.

\bibitem{CTQ1} {\sc T. B. Cao}, {\sc N. V. Thin} and {\sc S. D. Quang}, Difference analogues of the second main theorem for holomorphic curves and arbitrary families of hypersurfaces in projective varieties, \textit{Anal. Math.}, \textbf{50} (2024), 757-785.

\bibitem{CCP} {\sc T. Cao}, {\sc M. Cheng} and {\sc J. Peng} (2025), Some Malmquist type theorems on higher-order partial differential equations with delays, \textit{Complex Var. Elliptic. Equ.}, 1-25. https://doi.org/10.1080/17476933.2025.2464785.

\bibitem{CK} {\sc K. S. Charak} and {\sc R. Kumar}, Lappan's normality criterion in $\mathbb{C}^n$, \textit{Rend. Circ. Mat. Palermo}, II. Ser \textbf{72}, 239-252 (2023). https://doi.org/10.1007/s12215-021-00652-4.

\bibitem{PVD1} {\sc P. V. Dovbush}, Zalcman's lemma in $\mathbb{C}^n$, \textit{Complex Var. Elliptic Equ.}, \textbf{65} (5) (2020), 796-800.

\bibitem{PVD2} {\sc P. V. Dovbush}, Zalcman-Pang's lemma in $C^N$, \textit{Complex Var. Elliptic Equ.}, \textbf{66} (12) (2021), 1991-1997.

\bibitem{GL} {\sc Y. Gao} and {\sc K. Liu}, Meromorphic solutions of Bi-Fermat type partial differential and difference equations, \textit{Anal. Math. Phys.}, \textbf{14}, 130 (2024). https://doi.org/10.1007/s13324-024-00989-w.

\bibitem{GH} {\sc G. Haldar}, Solutions of Fermat-Type partial differential-difference equations in $\mathbb{C}^m$, \textit{Mediterr. J. Math.}, \textbf{20} (50) (2023). https://doi.org/10.1007/s00009-022-02180-6.

\bibitem{HZ1} {\sc W. Hao} and {\sc Q. Zhang}, Meromorphic solutions of a class of nonlinear partial differential equations, \textit{Indian J. Pure Appl. Math.}, 2025: 1-11, doi.org/10.1007/s13226-025-00779-5.

\bibitem{LH1} {\sc L. H\"{o}rmander}, An introduction to complex analysis in several variables, Van Nostrand. Princeton, N. J. 1966.

\bibitem{HY1} {\sc P. C. Hu} and {\sc C. C. Yang}, Uniqueness of meromorphic functions on $\mathbb{C}^m$, \textit{Complex var.}, \textbf{30} (1996), 235-270.

\bibitem{HLY1} {\sc P. C. Hu}, {\sc P. Li} and {\sc C. C. Yang}, Unicity of Meromorphic Mappings. Springer, New York (2003).

\bibitem{ps1} {\sc P. C. Hu} and {\sc C. C. Yang}, The Tumura-Clunie theorem in several complex variables, \textit{Bull. Aust. Math. Soc.}, \textbf{90} (2014), 444-456.

\bibitem{IM} {\sc O. A. Ivanova} and {\sc S. N. Melikhov}, A remark on holomorphic functions rational in some variables, \textit{Sib. Math. J.}, \textbf{65} (2024), 1112-1115.

\bibitem{K1} {\sc R. J. Korhonen}, A difference Picard theorem for meromorphic functions of several variables, \textit{Comput. Methods Funct. Theory}, \textbf{12} (1) (2012), 343-361.

\bibitem{LY11} {\sc B. Q. Li} and {\sc L. Yang}, Picard type theorems and entire solutions of certain nonlinear partial differential equations, \textit{J. Geom. Anal.}, (2025) 35:234, https://doi.org/10.1007/s12220-025-02067-4.

\bibitem{LS1} {\sc Y. Li} and {\sc H. Sun}, A note on unicity of meromorphic functions in several variables, \textit{J. Korean Math. Soc.}, \textbf{60} (4) (2023), 859-876.

\bibitem{FL1} {\sc F. L\"{u}}, Theorems of Picard type for meromorphic function of several complex variables, \textit{Complex Var. Elliptic Equ.}, \textbf{58} (8) (2013), 1085-1092.

\bibitem{FL2} {\sc F. L\"{u}}, On meromorphic solutions of certain partial differential equations, \textit{Canadian Math. Bull.}, 2025:1-15, doi:10.4153/S0008439525001.00,0.00,0.000347.

\bibitem{LZ} {\sc Z. Liu} and {\sc Q. Zhang}, Difference uniqueness theorems on meromorphic functions in several variables, \textit{Turk. J. Math.}, \textbf{42} (5) (2018), 2481-2505. 

\bibitem{MD1} {\sc S. Majumder} and {\sc P. Das}, Periodic behavior of meromorphic functions sharing values with their shifts in several complex variables, \textit{Indian J. Pure Appl. Math.}, https://doi.org/10.1007/s13226-025-00778-6.

\bibitem{MDP} {\sc S. Majumder}, {\sc P. Das} and {\sc D. Pramanik}, Sufficient condition for entire solution of a certain type of partial differential equation in $\mathbb{C}^m$, \textit{J. Contemp. Math. Anal.}, \textbf{60} (2025), 378-395.

\bibitem{MSP} {\sc S. Majumder}, {\sc N. Sarkar} and {\sc D. Pramanik}, Solutions of complex Fermat-type difference equations in several variables, \textit{Houston J. math.} (accepted for publication).


\bibitem{NW} {\sc J. Noguchi} and {\sc J. Winkelmann}, Nevanlinna theory in several complex variables and Diophantine approximation, Springer Tokyo Heidelberg New York Dordrecht London, 2013.

\bibitem{GS2} {\sc E. G. Saleeby}, On complex analytic solutions of certain trinomial functional and partial differential equations, \textit{Aequat. Math.}, \textbf{85} (2013), 553-562.

\bibitem{WS3} {\sc W. Stoll}, Value distribution and the lemma of the logarithmic derivative on polydiscs, \textit{Internat. J. Math. \& Math. Sci.}, \textbf{6} (4) (1983), 617-669.

\bibitem{WS2} {\sc W. Stoll}, Value distribution theory of meromorphic maps, Aspects Math., E7 (1985), 347, Vieweg-Verlag.

\bibitem{XC1} {\sc L. Xu} and {\sc T. Cao}, Solutions of complex Fermat-Type partial difference and differential-difference equations, \textit{Mediterr. J. Math.}, \textbf{15} (2018), 227.

\bibitem{XH} {\sc H. Y. Xu} and {\sc G. Haldar}, Entire solutions to Fermat-type difference and partial differential-difference equations in $\mathbb{C}^n$, \textit{Electronic J. Diff. Equ.}, \textbf{26} (2024), 1-21.

\bibitem{XLZ} {\sc H. Y. Xu}, {\sc K. Liu} and {\sc Z. Zuan}, Results on solutions of several product type nonlinear partial differential equations in $\mathbb{C}^3$, \textit{J. Math. Anal. Appl.}, \textbf{543} (1) (2025), 128-885.

\bibitem{XW1} {\sc H. Y. Xu} and {\sc H. Wang}, Notes on the existence of entire solutions for several partial differential-difference equations, \textit{Bull. Iran. Math. Soc.}, \textbf{47} (2020), 1477-1489. 

\bibitem{YH1} {\sc C. C. Yang} and {\sc X. H. Hua}, Uniqueness and value sharing of meromorphic functions, \textit{Ann. Acad. Sci. Fenn. Math.}, \textbf{22} (1997), 395-406.

\end{thebibliography}
\end{document}